\documentclass[reqno]{amsart}
\numberwithin{equation}{section}

\usepackage[colorlinks,linkcolor={blue}]{hyperref}

\usepackage{bm}
\usepackage{amssymb}
\usepackage[alphabetic,nobysame]{amsrefs}
\usepackage{graphicx} 
\usepackage{amscd}
\usepackage{enumerate}
\usepackage{cancel}
\usepackage[font=footnotesize]{caption}
\usepackage{dsfont}
\usepackage[colorinlistoftodos]{todonotes}
\usepackage{tikz-cd}
\usepackage[all]{xy}
\usepackage{mathtools}
\usepackage{scrextend}

\definecolor{darkred}{rgb}{1,0,0} 
\definecolor{darkgreen}{rgb}{0,0.8,0}
\definecolor{darkblue}{rgb}{0,0,1}

\hypersetup{colorlinks,
linkcolor=darkblue,
filecolor=darkgreen,
urlcolor=darkred,
citecolor=darkgreen}

\makeatletter
\providecommand\@dotsep{5}
\makeatother

\definecolor{orange}{RGB}{253,85,0}
\definecolor{darkgreen}{RGB}{0,95,10}

 \newcommand{\slp}{\alpha}
 \newcommand{\chat}{\widehat{c}}
 \newcommand{\chatEH}[1]{\widehat{c}^{\,\mathrm{EH}}_{#1}}

 \newcommand{\cEH}[1]{c^{\mathrm{EH}}_{#1}}
 \newcommand{\cV}[1]{c^{\mathrm{V}}_{#1}}
 \newcommand{\cGH}[1]{c^{\mathrm{GH}}_{#1}}
 
 \newcommand{\Z}{\mathds{Z}}
 \newcommand{\Q}{\mathds{Q}}
 \newcommand{\R}{\mathds{R}}
 \newcommand{\T}{\mathds{T}}
 \newcommand{\C}{\mathds{C}}
 \newcommand{\AAA}{\mathcal{A}}

 \newcommand{\FF}{\mathcal{F}}
 
 \newcommand{\UU}{\mathcal{U}}
 \newcommand{\VV}{\mathcal{V}}
 \newcommand{\WW}{\mathcal{W}}
 \newcommand{\KK}{\mathcal{K}}

 \newcommand{\HH}{\mathcal{H}}
 
 \newcommand{\DD}{\mathcal{D}}

 \newcommand{\Symp}{\mathrm{Symp}}
 
 \newcommand{\area}{\mathrm{area}}
 \newcommand{\crit}{\mathrm{crit}}
 \newcommand{\fix}{\mathrm{fix}}
 \newcommand{\id}{\mathrm{id}}
 
 \newcommand{\FR}{\mathrm{ind_{FR}}}
 \newcommand{\EH}{\mathrm{ind_{EH}}}
 
 \newcommand{\A}{\mathcal{A}}
 \newcommand{\aaa}{\bm{a}}
 \newcommand{\bbb}{\bm{b}}

 \DeclareMathOperator{\interior}{int}
 \DeclareMathOperator{\vol}{vol}

 \DeclareMathOperator*{\toup}{\longrightarrow} 
 \DeclareMathOperator*{\ttoup}{\llongrightarrow}

\DeclareRobustCommand{\llongrightarrow}{\relbar\joinrel\relbar\joinrel\rightarrow}

\DeclareMathOperator*{\bijection}{\longleftrightarrow}

 \theoremstyle{plain}
 \newtheorem{MainThm}{Theorem}
 \newtheorem{MainProp}[MainThm]{Proposition}
 
 \newtheorem{Question}[MainThm]{Question} 
 
 \newtheorem{Thm}{Theorem}[section]
 \newtheorem{Prop}[Thm]{Proposition}
 \newtheorem{Lemma}[Thm]{Lemma}

 \theoremstyle{definition}

 \newtheorem{Remark}[Thm]{Remark}
 \newtheorem{Example}[Thm]{Example}

\title[On the local maximizers of higher capacity ratios]{On the local maximizers of\\ higher capacity ratios}

\author[L. Baracco]{Luca Baracco}

\author[O. Bernardi]{Olga Bernardi}
\address{Luca Baracco and Olga Bernardi\newline\indent Università di Padova, Dipartimento di Matematica ``Tullio Levi Civita''\newline\indent Via Trieste 63, 35121 Padova, Italy}
\email{baracco@math.unipd.it, obern@math.unipd.it}

\author[C. Lange]{\\ Christian Lange}
\address{Christian Lange\newline\indent Ludwig-Maximilians-Universit\"at M\"unchen, Mathematisches Institut\newline\indent Theresienstraße 39, D-80333 Munich, Germany}
\email{lange@math.lmu.de}

\author[M. Mazzucchelli]{Marco Mazzucchelli}
\address{Marco Mazzucchelli\newline\indent CNRS, UMPA, \'Ecole Normale Sup\'erieure de Lyon\newline\indent 46 all\'ee d'Italie, 69364 Lyon, France}
\email{marco.mazzucchelli@ens-lyon.fr}

\thanks{Marco Mazzucchelli is partially supported by the ANR grants CoSyDy (ANRCE40-0014) and COSY (ANR-21-CE40-0002). Olga Bernardi and Marco Mazzucchelli are partially supported by the PRIN project 2017S35EHN 003 2019-2021 ``Regular and stochastic behaviour in dynamical systems''.}

\date{March 23, 2023}



\begin{document}

\begin{abstract}
We prove an analogue of the 4-dimensional local Viterbo conjecture for the higher Ekeland-Hofer capacities: 
on the space of 4-dimensional smooth star-shaped domains of unitary volume, endowed with the $C^3$ topology, the local maximizers of the $k$-th Ekeland-Hofer capacities are those domains symplectomorphic to suitable rational ellipsoids.

\tableofcontents 
\end{abstract}

\maketitle

\vspace{-30pt}

\section{Introduction}
\label{s:intro}

\subsection{The Viterbo conjecture}
Symplectic capacities are fundamental invariants that govern many rigidity phenomena in symplectic and contact topology. In this paper, we consider symplectic capacities $c$ defined on subsets $B$ of the symplectic vector space $(\C^{n},\omega)$, for $n\geq2$. If $B$ is a domain (i.e.~the closure of a non-empty open subset) of finite volume, its associated \emph{capacity ratios} are defined as
\begin{align*}
 \chat(B)=\frac{c(B)}{\vol(B,\omega)^{1/n}}.
\end{align*}
Here, the volume is the integral of $\omega^n$ on $B$. In his seminal paper \cite{Viterbo:2000aa}, Viterbo conjectured that $\chat(B)\leq\chat(B^{2n})$ for any convex body $B\subset\C^{n}$ and any symplectic capacity $c$, where $B^{2n}$ denotes any round ball in $\C^{n}$. This conjecture, which is currently open, is particularly hard as it combines symplectic invariants with the non-symplectic notion of convexity. Its validity would have far reaching consequences well beyond symplectic geometry: it would imply the validity of the long-standing Mahler conjecture from convex geometry, as it was pointed out by Artstein-Avidan, Karasev, and Ostrover \cite{Artstein-Avidan:2014aa}. 

By a $2n$-dimensional \emph{smooth convex body}, we mean a compact subset $B\subset\C^{n}$ with non-empty interior and smooth positively-curved boundary $\partial B$; for the sake of convenience, we shall always require that $B$ contains the origin in its interior. We consider the $C^k$ topology on the space of smooth convex bodies: two smooth convex bodies $B_0,B_1\subset\C^{n}$ are $C^k$-close when their boundaries $\partial B_0,\partial B_1$ are $C^k$-close submanifolds. 
A theorem of Abbondandolo and Benedetti \cite{Abbondandolo:2019aa}, extending an earlier result of Abbondandolo, Bramham, Hryniewicz, and Salomao \cite{Abbondandolo:2018aa} for the Ekeland-Hofer capacity in dimension 4, asserts the validity of a local version of the Viterbo conjecture: on a sufficiently $C^3$-small neighborhood of the round ball $B^{2n}$, the local maximizers of the function $B\mapsto\chat(B)$ are precisely the smooth convex bodies symplectomorphic to $B^{2n}$. In dimension 4, this theorem has been recently strengthen by Edtmair \cite{Edtmair:2022aa}: on a sufficiently $C^3$-small neighborhood of the round ball $B^{2n}$, all symplectic capacities coincide.

\subsection{Main result}\label{ss:result}
The goal of this paper is to investigate questions analogous to the (local and global)  Viterbo conjecture for higher symplectic capacities, whose definition is the following.
Each tuple $\aaa=(a_1,...,a_n)\in(0,\infty]^n$ defines a $2n$-dimensional domain
\begin{align*}
 E(\bm a) = \bigg\{ z=(z_1,...,z_n)\in\C^{n}\ \bigg|\  \sum_{j=1}^n \frac{|z_j|^2}{a_j}=\frac1\pi \bigg\}.
\end{align*}
If all the entries $a_i$ are finite, $E(\aaa)$ is an ellipsoid.
A \emph{symplectic $k$-capacity} on $\C^{n}$ is a function $c_k:2^{\C^{n}}\to[0,\infty]$, where $2^{\C^{n}}$ is the collection of subsets of $\C^{n}$, satisfying the following three properties:
\begin{itemize}
\setlength\itemsep{5pt}

\item \textbf{(Monotonicity)} $c_k(U)\leq c_k(V)$ if there exists $\psi\in\mathrm{Symp}(\C^{n},\omega)$ such that $\psi(U)\subset V$.

\item \textbf{(Conformality)} $c_k(b\,U)=b^2 c_k(U)$ for each $b\in\R$ and $U\subset\C^{n}$.

\item \textbf{(Normalization)} For each $\aaa\in(0,\infty]^n$ with $a_1<\infty$, if we enumerate the set of positive multiples of the entries $a_j$ in increasing order as 
$\tau_1<\tau_2<\tau_3<...$,
the value $c_k(E(\aaa))$ is the $k$-th  element in the sequence
\begin{align*}
\underbrace{\tau_1,...,\tau_1}_{\times m_1},\underbrace{\tau_2,...,\tau_2}_{\times m_2},\underbrace{\tau_3,...,\tau_3}_{\times m_3},...
\end{align*}
where $m_h$ is the number of parameters $a_j$ having $\tau_h$ as a positive multiple.

\end{itemize}

Ordinary symplectic capacities on $\C^{n}$, as appearing in the Viterbo conjecture, are 1-capacities\footnote{In the literature, and in particular in the statement of the Viterbo conjecture, a stronger form of monotonicity is sometimes assumed for the symplectic capacities: $c(U)\leq c(V)$ if there exists a symplectic embedding of $U$ into $V$. Capacities and, more generally, $k$-capacities satisfying the weaker monotonicity stated above are sometimes called relative or non-intrinsic \cite[Sect.~12.1]{McDuff:2017aa}.}. While there are many 1-capacities in the literature, to the best of the authors' knowledge there are only few constructions of symplectic $k$-capacities: the  Ekeland-Hofer capacities $\cEH k$ \cite{Ekeland:1989aa, Ekeland:1990aa} defined via an infinite dimensional linking argument (see Section~\ref{ss:def_EH_capacities}), the Viterbo capacities $\cV k$ \cite{Viterbo:1989aa} defined via finite dimensional reductions, the Gutt-Hutchings capacities $\cGH k$ \cite{Gutt:2018aa} defined via $S^1$-equivariant symplectic homology, and the very recent Zhang capacities \cite{Zhang:2021aa} defined via the microlocal theory of sheaves. At least the three $k$-capacities $\cEH k$, $\cV k$, and $\cGH k$ are expected to coincide\footnote{The equality $\cEH k=\cGH k$ was recently announced by Gutt and Ramos \cite{Gutt:2023aa}.}.

We say that $B\subset\C^n$ is a $2n$-dimensional \emph{smooth star-shaped domain} when it is diffeomorphic to a compact $2n$-dimensional ball and has a smooth boundary $\partial B$ that is transverse to all radial lines on $\C^n$.  Notice that the smooth convex bodies, and in particular the ellipsoids, are  smooth star-shaped domains. Once again, we consider the $C^k$ topology on the space of smooth star-shaped domains: $B_0$ and $B_1$ are $C^k$-close when their boundaries $\partial B_0,\partial B_1$ are $C^k$-close submanifolds.

An ellipsoid $E=E(a_1,...,a_n)$ is called \emph{rational} when all the ratios $a_j/a_h$ are rational numbers. In this case, we denote by $\tau(E)$ the least common multiple of the parameters $a_1,...,a_n$. For each integer $m\geq1$, we define $k_m(E)$ to be the minimal integer $k\geq1$ such that $c_k(E)=m\,\tau(E)$. Here, $c_k$ is any symplectic $k$-capacity, and the values $k_m(E)$ are clearly independent of the choice of such $c_k$'s. Finally, we set $K(E):=\{k_m(E)\ |\ m\geq1\}$.
Our main result is a confirmation of a version of the local Viterbo conjecture for the higher Ekeland-Hofer capacities in dimension 4.

\begin{MainThm}
\label{mt:local_max}
On the space of $4$-dimensional smooth star-shaped domains endowed with the $C^3$ topology, the local maximizers of $B\mapsto \chatEH k(B)$ are precisely those domains symplectomorphic to a $4$-dimensional rational ellipsoid $E$ with $k\in K(E)$.
\end{MainThm}

\subsection{Open questions on higher capacity ratios}
In the Viterbo conjecture, the convexity assumptions on the domains cannot be removed: in an unpublished note \cite{Hermann:1998aa}, Hermann showed that, for any symplectic 1-capacity $c_1$ on $\C^n$, the associated capacity ratio $\chat_1$ is unbounded over the space of star-shaped domains. In contrast, 
a theorem of Artstein-Avidan, Milman, and Ostrover \cite{Artstein-Avidan:2008aa} asserts that any 1-capacity ratio is uniformly bounded over the space of convex bodies by a constant that is independent of the dimension (previously, Viterbo \cite{Viterbo:2000aa} showed that there is a uniform bound growing linearly in the dimension). A minor modification in their proof provides the following uniform bound for all $k$-capacity ratios.

\begin{Thm}[Artstein-Avidan, Milman, Ostrover]\label{t:AMO}
There exists a constant $a>0$ such that, for each integer $n\geq1$, for each symplectic $k$-capacity $c_k$ on $\C^{n}$, and for each $2n$-dimensional convex body $B$, we have
$\chat_k(B)\leq ak$.
\end{Thm}

\begin{proof}
By \cite[Theorem 1.6]{Artstein-Avidan:2008aa}, there exists a constant $a>0$ such that, for any $n\geq1$ and for any convex body $B\subset\C^{n}$, we have
\begin{align*}
\frac{c_{\mathrm{lin}}(B)}{\vol(B)^{1/n}}
\leq 
a,
\end{align*}
where  
\begin{align*}
c_{\mathrm{lin}}(B)
:=
\inf\big\{b>0\ \big|\ \psi(B)\subset \underbrace{E(b)\times\C^{n-1}}_{=E(b,\infty,...,\infty)}\mbox{ for some }\psi\in\mathrm{Sp}(2n)\big\}.
\end{align*}
The normalization of the symplectic $k$-capacity implies $c_k(E(b)\times\C^{n-1})=kb$, and together with the monotonicity property we obtain $c_k(B)\leq k\,c_{\mathrm{lin}}(B)$. Therefore 
\[
\chat_k(B)
\leq
k \frac{ c_{\mathrm{lin}}(B)}{\vol(B)^{1/n}}
\leq
ak .
\qedhere
\]
\end{proof}

In view of Theorem~\ref{t:AMO}, it is natural to raise the following question, which for 1-capacities would be answered by the Viterbo conjecture if confirmed. 

\begin{Question}
For any symplectic $k$-capacity $c_k$, what are the global maximizers of the function $B\mapsto\chat_k(B)$ over the space of convex bodies in $\C^n$?
\end{Question}

Simple algebraic computations allow to detect the local and global maximizers of $\chat_k$ over the space of ellipsoids.

\begin{MainProp}\label{mp:ellipsoids}
Let $c_k$ be any symplectic $k$-capacity on $\C^n$.
\begin{itemize}
\setlength\itemsep{3pt}

\item[$(i)$] The local maximizers of the capacity ratio $\chat_k$ on the space of $2n$-dimensional ellipsoids are precisely those rational ellipsoids $E$ such that $k\in K(E)$.

\item[$(ii)$] The global maximum of the capacity ratio $\chat_k$ on the space of $2n$-dimensional ellipsoids is
\begin{align*}
 \max_{\aaa}\ \chat_k(E(\aaa))
 =
 (q+1)^{\frac{n-r+1}{n}} (q+2)^{\frac{r-1}{n}},
\end{align*}
where  $q\geq0$ and $r\in\{1,...,n\}$ are integers such that $k=q n + r$. The global maximizers are precisely those rational ellipsoids $E(\aaa)$ of the form
\begin{align*}
\frac{a_i}{a_n}=
\left\{
  \begin{array}{@{}ll}
    \frac{q+1}{q+2}, & \mbox{if $r>1$ and $1\leq i\leq r-1$,}\vspace{5pt} \\ 
    1 & \mbox{if $r\leq i\leq n$.}  
  \end{array}
\right.
\end{align*}

\end{itemize}
\end{MainProp}

One might be tempted to conjecture that each capacity ratio $\chat_k$ achieves its maximum over the space of ellipsoids. While for $k=1$ this is implied by the Viterbo conjecture, for large $k$ this turns out to be false! 

\begin{Example}
\label{ex:poly_vs_ellipsoids}
Proposition~\ref{mp:ellipsoids} implies that the 4-dimensional ellipsoid maximizing any $k$-capacity ratio $\chat_k$ is $E(\big\lceil\tfrac{k}{2}\big\rceil,\big\lceil\tfrac{k+1}{2}\big\rceil)$, and 
\begin{align*}
 \chat_k\big(E(\big\lceil\tfrac{k}{2}\big\rceil,\big\lceil\tfrac{k+1}{2}\big\rceil)\big)
 =
 \sqrt{\big\lceil\tfrac{k}{2}\big\rceil\big\lceil\tfrac{k+1}{2}\big\rceil},
\end{align*}
where $\lceil x\rceil=\min\{y\in\Z\ |\ y\geq x\}$ denotes the ceiling function.
For any $k$-capacity $c_k\in\{\cEH k,\cV k,\cGH k\}$,  the polydisks $P(a,b)=E(a)\times E(b)$ have capacity ratios
\begin{align*}
 \chat_k(P(a,b))
 =
 k\frac{\min\{a,b\}}{\sqrt{2ab}},
\end{align*}
as it was computed in \cite{Ekeland:1990aa, Moatty:1994aa, Gutt:2018aa}, see Example~\ref{ex:polydisk}; in particular $P(1,1)$ maximizes all capacity ratios $\chat_k$ among 4-dimensional polydisks: 
\[\chat_k(P(1,1))
= 
 \frac{k}{\sqrt2}.\]
For $k=2$, the maximizer polydisk $P(1,1)$ and the maximizer ellipsoid $E(1,2)$ have the same capacity ratio 
\[\chat_2(P(1,1))=\chat_2(E(1,2))=\sqrt2,\] 
whereas for any $k\geq 3$ we have 
\[
\tag*{\qed}
\chat_k(P(1,1))>\chat_k\big(E(\big\lceil\tfrac{k}{2}\big\rceil,\big\lceil\tfrac{k+1}{2}\big\rceil)\big),\qquad  \forall k\geq3.\]
\end{Example}

Example~\ref{ex:poly_vs_ellipsoids} suggests that $P(1,1)$ may have relevant maximality properties for the higher capacity ratios. This is the case within a suitable class of toric domains in $\C^n$, which we now introduce. Let $S^1\subset\C$ be the unit circle in the complex plane, and consider the action of the torus $\T^n=S^1\times...\times S^1$ on $\C^n$ given by
\begin{align*}
 (e^{i\theta_1},...,e^{i\theta_n})\cdot(z_1,...,z_n)
 =
 (e^{i\theta_1}z_1,...,e^{i\theta_n}z_n).  
\end{align*}
A toric domain is a $\T^n$-invariant subset of $\C^n$ given by the closure of some non-empty open set. Any domain $\Omega\subset[0,\infty)^n$ is the profile of an associated toric domain 
\[X_\Omega:=\big\{(z_1,...,z_n)\in\C^n\ \big|\ (\pi|z_1|^2,...,\pi|z_n|^2)\in\Omega\big\}.\]
Two classes of toric domains are particularly relevant in symplectic geometry:
\begin{itemize}
\setlength\itemsep{5pt}

\item The \emph{convex} toric domains $X_\Omega$, which are those toric domains such that $\{(x_1,...,x_n)\in\R^n\ |\ (|x_1|,...,|x_n|)\in\Omega\}$ is compact and convex.

\item The \emph{concave} toric domains $X_\Omega$, which are those toric domains whose profile $\Omega$ is compact and such that $[0,\infty)^n\setminus\Omega$ is convex.

\end{itemize}
Polydisks $P(\aaa)=E(a_1)\times...\times E(a_n)$ are examples of convex toric domains, and ellipsoids $E(\aaa)$ are the only toric domains that are simultaneously convex and concave. We stress that ``convex'' and ``concave'' here do not refer to the convexity of $X_\Omega$ as a subset of $\C^n$: all convex toric domains are convex subsets of $\C^n$, but some concave toric domains are convex subsets of $\C^n$ as well. In his Ph.D. thesis \cite{Moatty:1994aa}, Moatty provided combinatorial formulas to compute the Viterbo capacities $\cV k$ of convex or concave toric domains. Analogous formulas were provided by Gutt and Hutchings \cite{Gutt:2018aa} for their capacities $\cGH k$. 

If $B\subset \C^n$ is a smooth star-shaped domain, its boundary $\partial B$ admits a canonical contact form $\lambda|_{\partial B}$, which is the restriction of the 1-form on $\C^n$ given by
\begin{align}
\label{e:Liouville}
\lambda=\frac12\sum_{j=1}^{n} \big( x_jdy_j-y_jdx_j \big).
\end{align}
Here, $x_1,y_1,...,x_n,y_n$ are the standard Darboux coordinates on $\C^n$. As pointed out in \cite{Gutt:2018aa}, the formulas for convex and concave toric domains hold for all symplectic $k$-capacities satisfying the following extra condition:

\begin{itemize}

\item \textbf{(Closed Reeb orbits)} For all smooth star-shaped domains $B\subset\C^n$ such that $(\partial B,\lambda)$  has non-degenerate closed Reeb orbits, $c_k(B)$ is the period of a closed Reeb orbit with Conley-Zehnder index $2k+n-1$.

\end{itemize}

Moatty-Gutt-Hutchings' formulas may be employed to study the global maximizers of the higher capacity ratios over convex and concave toric domains. The combinatorics becomes rather involved. We worked out the details in dimension 4.

\begin{MainProp}
\label{mp:toric}
For any symplectic $k$-capacity on $\C^2$ satisfying the extra ``closed Reeb orbits'' assumption, the following points hold.
\begin{itemize}
\setlength\itemsep{3pt}

\item[$(i)$] On the space of $4$-dimensional concave toric domains, $\chat_k$ achieves its global maximum on the ellipsoid $E\big(\big\lceil\frac k2\big\rceil,\big\lceil\frac {k+1}2\big\rceil\big)$.

\item[$(ii)$] On the space of $4$-dimensional convex toric domains, $\chat_k$ achieves its maximum on:
\begin{itemize}
\item[$\bullet$]  the round balls $E(a,a)$, $a>0$, if $k=1$,
\item[$\bullet$] the ellipsoids $E(a,2a)$ and the polydisks $P(a,a)$, $a>0$, if $k=2$,
\item[$\bullet$] the polydisks $P(a,a)$, $a>0$, if $k\geq 3$.
\end{itemize}
\end{itemize}
\end{MainProp}

The situation is largely unexplored in higher dimensions. Let $P_{2n}:=P(1,...,1)$ be the $2n$-dimensional polydisk with unitary parameters, and $E_{2n}$ the global maximizer of $\chat_k$ over the space of $2n$-dimensional ellipsoids (see Proposition~\ref{mp:ellipsoids}(ii)). Simple computations show that there exists a minimal integer $\kappa(n)\geq2$ such that 
\[\chat_k(P_{2n})\geq\chat_k(E_{2n}),\qquad \forall k\geq\kappa(n).\] 
Moreover, $\kappa(n)\to\infty$ as $n\to\infty$. This raises the following questions.

\begin{Question}
For each dimension $2n\geq4$ and for each $k\geq\kappa(n)$, is the polydisk $P_{2n}$ a global maximizers of the $k$-capacity ratios $\chat_k$ over the space of $2n$-dimensional convex toric domains?
\end{Question}

\begin{Question}
In some dimension $2n\geq4$, do there exist arbitrarily large integers $k$ and $2n$-dimensional convex bodies $B\subset\C^n$ such that $\chat_k(B)>\chat_k(P_{2n})$?
\end{Question}

\subsection{Outline of the proof of Theorem~\ref{mt:local_max}}
The local Viterbo conjecture was established as a consequence of a systolic characterization of Zoll contact manifolds, which we now briefly recall. Let $(N,\alpha)$ be a closed contact manifold, $X$ its Reeb vector field defined by $\alpha(X)\equiv1$ and $d\alpha(X,\cdot)\equiv0$, and $\phi^t:N\to N$ its Reeb flow. The Weinstein conjecture, which is a theorem for several classes of closed contact manifolds including all 3-dimensional ones, asserts that the Reeb flow $\phi^t$ has at least one closed  orbit. If this is the case, we denote by $\tau_1(N,\alpha)=\inf\{t>0\ |\ \fix(\phi^t)\neq\varnothing\}$ the systole, namely the minimum among the periods of the closed Reeb orbits; if the Weinstein conjecture fails for $(N,\alpha)$, we instead set $\tau_1(N,\alpha)=0$. We consider the associated systolic ratio
\begin{align}
\label{e:systolic_ratio}
 \widehat \tau_1(N,\alpha)=\frac{\tau_1(N,\alpha)}{\vol(N,\alpha)^{1/n}},
\end{align}
where $n=\tfrac12(\dim(N)+1)$, and $\vol(N,\alpha)$ is the integral of $\alpha\wedge (d\alpha)^{n-1}$.
On the space of contact forms on $N$ endowed with the $C^3$-topology, the local maximizers of the function $\alpha\mapsto\widehat\tau_1(N,\alpha)$ are precisely the \emph{Zoll} contact forms: those contact forms whose associated Reeb orbits are all closed and have the same minimal period (namely $\phi^{\tau_1(N,\alpha)}=\id$ and $\fix(\phi^t)=\varnothing$ for all $t\in(0,\tau_1(N,\alpha))$). This remarkable theorem was established in full generality by Abbondandolo and Benedetti \cite{Abbondandolo:2019aa}, and previously in dimension 3 in a series of papers \cite{Alvarez-Paiva:2014aa, Abbondandolo:2018aa, Benedetti:2021aa}. 

A closed contact manifold $(N,\alpha)$, or just the contact form $\alpha$, is called \emph{Besse} when all its Reeb orbits are closed, and in this case they automatically have a common period according to a theorem of Wadsley \cite{Wadsley:1975aa}. The simplest examples of Besse closed contact manifolds are the boundaries of the rational ellipsoids. Actually, in dimension 4, for each smooth star-shaped domain $B\subset\C^2$ with a Besse boundary there exists a symplectomorphism $\psi\in\Symp(\C^n,\omega)$ such that $\psi(B)$ is a rational ellipsoid; this is a consequence of the classification of 3-dimensional Besse contact spheres \cite[Theorem~1.1]{Mazzucchelli:2023aa}, together with a straightforward generalization of \cite[Prop.~4.3]{Abbondandolo:2018aa}.

Assume now that $(N,\alpha)$ is a closed contact 3-manifold. For each $k\geq1$, we denote by $\tau_k(N,\alpha)$ the infimum of the values $\tau>0$ such that there exist at least $k$ closed Reeb orbits (counting iterates as well) of period less than $\tau$; in formulas:
\begin{align*}
 \tau_k(N,\alpha):=\inf\bigg\{\tau>0\ \bigg|\ \sum_{t\in(0,\tau)} \#\big(\fix(\phi^t)/\!\sim\!\big)\geq k\bigg\},
\end{align*}
where $\sim$ is the equivalence relation identifying $z\sim \phi^t(z)$ for all $z\in N$ and $t\in\R$. Extending the notation~\eqref{e:systolic_ratio}, we write
\begin{align*}
 \widehat\tau_k(N,\alpha) := \frac{\tau_k(N,\alpha)}{\vol(N,\alpha)^{1/2}}.
\end{align*}
If $(N,\alpha)$ is Besse, we denote by $k_1:=k_1(N,\alpha)$ the minimal positive integer such that $\tau_{k_1}$ is a common period of the Reeb orbits, i.e.~$\phi^{\tau_{k_1}}=\id$. This notation is consistent with the one employed in Section~\ref{ss:result} with 4-dimensional rational ellipsoids $E$: we have $k_1(E)=k_1(\partial E,\alpha)$. 
A result of Abbondandolo, Lange, and Mazzucchelli \cite{Abbondandolo:2022aa} generalized the above mentioned systolic characterization of Zoll contact 3-manifolds as follows: on the space of contact forms on a closed 3-manifold $N$, endowed with the $C^3$ topology, the local maximizers of the function $\alpha\mapsto\widehat\tau_k(N,\alpha)$ are precisely the Besse contact forms $\alpha$ with $k_1(N,\alpha)=k$.

Consider now a $2n$-dimensional smooth star-shaped domain $B$. Stokes theorem  implies that $\vol(B,\omega)=\vol(\partial B,\lambda)$. Moreover, if $B$ is a smooth convex body and $c_1$ denotes any of the above mentioned 1-capacities $\cEH 1$, $\cV 1$, or $\cGH 1$, we have $c_1(B)=\tau_1(\partial B,\lambda)$,  
as it was proven by Sikorav \cite{Sikorav:1990aa}, Viterbo \cite{Viterbo:1989aa}, and Abbondandolo and Kang \cite{Abbondandolo:2022ab} for the three respective capacities. Therefore $\chat_1(B)=\widehat\tau_1(\partial B,\lambda)$, and the systolic Zoll characterization implies the local Viterbo conjecture.

The proof of Theorem~\ref{mt:local_max} employs the Clarke action functional associated with a smooth convex body $B$, whose set of critical values is precisely the action spectrum $\sigma(\partial B)$, that is, the set of periods of the closed Reeb orbits of $\partial B$. Ekeland and Hofer introduced spectral invariants $s_k(B)$, which are critical values of the Clarke action functional selected by suitable min-max procedures. A result of Ginzburg, G\"urel, and Mazzucchelli \cite{Ginzburg:2021aa} asserts that a $2n$-dimensional smooth convex body $B_0$ has Besse boundary if and only if $s_k(B_0)=s_{k+n-1}(B_0)$ for some $k\geq1$.  As a first crucial ingredient for the proof of Theorem~\ref{mt:local_max}, we show that, for any smooth convex body $B_1$ sufficiently $C^2$-close to such a  $B_0$, the spectral invariant $s_{k_1}(B_1)$ is the minimum among the element in the action spectrum $\sigma(\partial B_1)$ that are close to $s_{k_1}(B_0)$ (Lemma~\ref{l:perturbation}).

Assume now that $B_0$ as above has dimension $2n=4$.
We denote by $K(B_0)$ the collection of  values $k$ such that $s_k(B_0)=s_{k+1}(B_0)$, and denote by $k_m=k_m(B_0)$ the $m$-th smallest element of $K(B_0)$. Once again, this notation agrees with the one employed in Section~\ref{ss:result} for the rational ellipsoids. For each $m\geq1$, Lemma~\ref{l:perturbation} together with some Morse theory for the Clarke action functional implies that $s_{k_m}(B_1)\leq m\tau_{k_1}(B_1)$ for all smooth convex bodies $B_1$ sufficiently $C^2$-close to $B_0$, and equality holds if $B_1=B_0$ (Lemma~\ref{l:s_k=tau_k}). Therefore, by the above mentioned result of Abbondandolo, Lange, and Mazzucchelli \cite{Abbondandolo:2022aa}, we conclude in Theorem~\ref{t:main_Clarke} that any such $B_0$ is a local maximizer, with respect to the $C^3$ topology, of the spectral ratios
\begin{align*}
 \widehat s_{k_m}(B):=\frac{s_{k_m}(B)}{\vol(B)^{1/2}}.
\end{align*}

It is well known that $s_1(B)=\tau_1(\partial B,\lambda)=\cEH 1(B)$ for all $2n$-dimensional smooth convex bodies $B$.
Conjecturally, $s_k(B)=\cEH k(B)$ for all $k\geq2$ as well. We shall show that at least the inequality $\cEH k(B)\leq s_k(B)$ holds for all $k\geq2$ (Proposition~\ref{p:capacity<spectral}). Moreover, the equality $\cEH k(B)= s_k(B)$ holds for all $k\geq2$ and for all 4-dimensional Besse convex bodies, using the already mentioned fact that every such $B$ is symplectomorphic to a rational ellipsoid. By employing these facts together with Theorem~\ref{t:main_Clarke}, we infer that the Besse $4$-dimensional  smooth convex bodies $B_0$ with associated integers $k_m=k_m(B_0)$ are local maximizers, with respect to the $C^3$ topology, of the capacity ratios $\chatEH k$.

Finally, building on a well-known argument originally due to Alvarez Paiva and Balacheff \cite{Alvarez-Paiva:2014aa}, we show that any local maximizer $B$ of the capacity ratios $\chatEH k$ over the space of $2n$-dimensional smooth star-shaped domains must have Besse boundary, and the Reeb orbits therein must have common period $c_k(B)$. In particular every such $B$ is symplectomorphic to a rational ellipsoid $E$ that is a local maximizer of $\chatEH k$, and our sharp result for ellipsoids (Proposition~\ref{mp:ellipsoids}(i)) implies that $k\in K(E)$.

\subsection{Organization of the paper}
In Section~\ref{s:Clarke}, after recalling the general properties of the Clarke action functional, we prove that Besse convex bodies are local maximizers of suitable spectral ratios (Theorem~\ref{t:main_Clarke}). In Section~\ref{s:EH} we prove that the Ekeland-Hofer capacities of smooth convex bodies are bounded from above by the corresponding spectral invariants of the Clarke action functional (Proposition~\ref{p:capacity<spectral}), and employ this result and Theorem~\ref{t:main_Clarke} to prove Theorem~\ref{mt:local_max}. The proof of Proposition~\ref{p:capacity<spectral} requires some subtle properties of the Fadell-Rabinowitz index that are either stated under slightly different assumptions or not explicitly stated in the literature, and we included the details in Appendix~\ref{a:Fadell_Rabinowitz}. In Section~\ref{s:ellipsoids} we prove Proposition~\ref{mp:ellipsoids}. Finally, in Section~\ref{s:toric}, we prove Proposition~\ref{mp:toric}.

\section{The Clarke action functional}\label{s:Clarke}

\subsection{Smooth convex bodies}
\label{ss:convex_contact_spheres}
Let $B\subset\C^n$ be a $2n$-dimensional smooth convex body, and as usual we assume that $B$ contains the origin in its interior. We consider the Hamiltonian $h:\C^2\to[0,\infty)$ that is 2-homogeneous and is identically equal to 1 on the boundary $\partial B$, i.e.
\begin{align*}
 h(cz)=c^2,\qquad\forall c\geq0,\ z\in\partial B.
\end{align*}
Its Hamiltonian vector field is defined as usual by $\omega(X_h,\cdot)=dh$, where $\omega=d\lambda$ is the standard symplectic form of $\C^{n}$, and $\lambda$ is its primitive~\eqref{e:Liouville}. The Hamiltonian flow $\phi_h^t$ is 1-homogeneous, i.e.\
$\phi_h^t(cz)=c\,\phi_h^t(z)$,
and its restriction $\phi_h^t|_{\partial B}$ coincides with the Reeb flow of $\partial B$ (with respect to the canonical contact form $\lambda|_{\partial B}$). In particular, there is a one-to-one correspondence between closed Reeb orbits of $\partial B$ and cylinders of periodic orbits of $\phi_h^t$. The action spectrum $\sigma(\partial B)$ is the set of periods of the closed Reeb orbits, i.e.
\begin{align*}
 \sigma(\partial B)=\big\{ t>0\ \big|\ \fix(\phi_h^t)\neq\varnothing \big\}.
\end{align*}

Since the Hamiltonian  $h$ has positive definite Hessian everywhere outside the origin, it admits a Legendre dual  $h^*:\C^n\to[0,\infty)$ given by
\begin{align}
\label{e:dual_Hamiltonian}
 h^*(w)=\max_{z\in\C^n} \Big( \langle w,z\rangle - h(z)\Big),
\end{align}
which is also 2-homogeneous.
We set $S^1=\R/\Z$, and consider the Hilbert space
\[
L^2_0(S^1,\C^n) 
=
\left\{ 
u\in L^2(S^1,\C^n)\ 
\ \left|\ 
\int_{S^1}u(t)\,dt=0 
\right.\right\}. \]
Its elements are precisely the maps of the form $u=\dot\gamma$, where $\gamma\in W^{1,2}(S^1,\C^n)$.
We consider the functionals $\AAA:L^2_0(S^1,\C^n)\to\R$ and $\HH:L^2_0(S^1,\C^n)\to[0,\infty)$, given by
\begin{align*}
 \AAA(\dot\gamma)=\frac12\int_{S^1} \langle J\gamma(t),\dot\gamma(t) \rangle\,dt,
 \qquad
 \HH(\dot\gamma)=\int_{S^1} h^*(-J\dot\gamma(t))\,dt.
\end{align*}
Notice that the expression of $\AAA$ involves a primitive of $\dot\gamma$, but is actually independent of the choice of such a primitive. The functional $\AAA$ is a non-degenerate quadratic form, and we consider its positive open cone $\AAA^{-1}(0,\infty)$.
The circle $S^1$ acts on it by translation, i.e.~$t\cdot u=u(t+\cdot)$ for all $t\in S^1$ and $u\in\AAA^{-1}(0,\infty)$, and both $\AAA$ and $\HH$ are $S^1$-invariant. 

The \emph{Clarke action functional} $\widetilde\Psi:\AAA^{-1}(0,\infty)\to(0,\infty)$ is defined by
\begin{align*}
 \widetilde\Psi(u)=\frac{\HH(u)}{\AAA(u)}.
\end{align*}
Let us summarize its properties (and refer the reader to Ekeland-Hofer's \cite{Ekeland:1987aa} for a proof of the non-trivial ones):
\begin{itemize}
\setlength\itemsep{5pt}

\item \textbf{($\bm{\C_*}$-invariance)} It is $0$-homogeneous and $S^1$-invariant, i.e.\ $\widetilde\Psi(cu)=\widetilde\Psi(u)=\widetilde\Psi(t\cdot u)$ for all $c>0$, $t\in S^1$, and $u\in\AAA^{-1}(0,\infty)$. Overall, $\widetilde\Psi$ is $\C_*$-invariant, where $\C_*=\C\setminus\{0\}\equiv(0,\infty)\times S^1$.

\item \textbf{(Regularity)} It has the same regularity as $\HH$: it is $C^{1,1}$ and admits a Hessian in the sense of Gateaux at every point.

\item \textbf{(Clarke variational principle)} There is a one-to-one correspondence 
\begin{align*}
\crit(\widetilde\Psi)\cap\widetilde\Psi^{-1}(\tau)&\ \bijection^{1:1}\ \fix(\phi_h^\tau)\\
u_z&\ \bijection\  
z
\end{align*}
where
$u_z(t)=\tfrac d{dt} \phi_H^{\tau t}(z)$ and, conversely, $z=\tau^{-1}\nabla h^*(-Ju_z(0))$.

\item \textbf{(Morse indices)} The Morse index and nullity\footnote{Here, we keep into account the $\C_*$-invariance of $\widetilde\Psi$: the nullity of a critical point $u$ of $\widetilde\Psi$ is defined as $\dim\ker(d^2\widetilde\Psi(u))-2$.} of any critical point $u_z\in\crit(\widetilde\Psi)\cap\widetilde\Psi^{-1}(\tau)$ are finite. Indeed, the Morse index is equal (up to a conventional additive constant) to the Maslov index of the associated $\tau$-periodic orbit $t\mapsto\phi_h^t(z)$, and the nullity is equal to $\dim\ker(d\phi_h^\tau(z)-I)-2$.

\item \textbf{(Palais-Smale condition)} Being $0$-homogeneous, $\widetilde\Psi$ cannot satisfy the Palais-Smale condition. Nevertheless, let $V$ be the rescaled version of the anti-gradient of $\widetilde\Psi$ given by
\begin{align*}
 V(u):=-\|u\|_{L^2}\nabla\widetilde\Psi(u),
\end{align*}
whose flow $g_t:\AAA^{-1}(0,\infty)\to\AAA^{-1}(0,\infty)$ is $1$-homogeneous and preserves the $L^2$ norm, i.e.\ $\tfrac d{dt}\|g_t(u)\|_{L^2}=0$. On an $L^2$-sphere of any radius $r>0$, the Clarke action functional $\widetilde\Psi$ satisfies the Palais-Smale condition with respect to $V$: any sequence $u_k\in\AAA^{-1}(0,\infty)$ such that $V(u_k)\to0$, $\|u_k\|_{L^2}=r$, and $\sup\widetilde\Psi(u_k)<\infty$ admits a converging subsequence.

\end{itemize}

Due to its $0$-homogeneity, in the literature the Clarke action functional $\widetilde\Psi$ usually appears restricted to $S^1$-invariant hypersurfaces of $\AAA^{-1}(0,\infty)$ transverse to the radial directions: for instance $\AAA^{-1}(1)$ or $\HH^{-1}(1)$. In this paper, we equivalently restrict $\widetilde\Psi$ to the smooth hypersurface 
$\Lambda=\big\{ u\in\AAA^{-1}(0,\infty)\ \big|\ \|u\|_{L^2}=1  \big\}$,
and denote the restriction by \[\Psi:=\widetilde\Psi|_\Lambda.\] We still call $\Psi$ the Clarke action functional. 
Clearly, $\Psi$ is $S^1$-invariant, and we have $S^1\cdot u\subset\crit(\Psi)\cap\Psi^{-1}(\tau)$ if and only if $\C_*\cdot u\subset\crit(\widetilde\Psi)\cap\widetilde\Psi^{-1}(\tau)$. 
Overall, $\Psi$ satisfies the classical properties required by $S^1$-equivariant Morse theory, except for its lack of $C^2$ regularity, which nevertheless will be easily circumvented later on.

\subsection{Spectral invariants}\label{ss:Clarke_spectral_invariants}
In their seminal work \cite{Ekeland:1987aa}, Ekeland and Hofer introduced and studied spectral invariants for the Clarke action functional. We briefly recall the construction.

Consider the subspace $\Lambda_+$ given by all $u\in\Lambda$  having the form
\begin{align*}
 u(t)=\sum_{k=1}^\infty e^{i2\pi k t}u_k,
\end{align*}
for some $u_k\in\C^n$. Notice that $\Lambda_+$ is the unit sphere of an infinite dimensional Hilbert subspace of $L^2_0(S^1,\C^n)$, and is $S^1$-invariant. Since $\Lambda_+$ is contractible, the  Gysin sequence of the $S^1$-bundle $\Lambda_+\times ES^1\to \Lambda_+\times_{S^1} ES^1$ readily implies that its $S^1$-equivariant cohomology with rational coefficients is given by $H^*_{S^1}(\Lambda_+;\Q)=\Q[e]$, where $e$ is a generator of $H^2_{S^1}(\Lambda_+;\Q)\cong\Q$. From now on, all cohomology rings will be assume to have rational coefficients, and we will suppress $\Q$ from the notation. The inclusion $\Lambda_+\hookrightarrow\Lambda$ is an $S^1$-equivariant homotopy equivalence, and therefore induces a ring isomorphism in $S^1$-equivariant cohomology. Summing up, we have
\[H^*_{S^1}(\Lambda)=\Q[e],\]
where $e$ now denotes a generator of $H^2_{S^1}(\Lambda)\cong\Q$. In particular, $e^k\neq0$ in $H^*_{S^1}(\Lambda)$ for all $k\geq0$; in terms of the Fadell-Rabinowitz index (see Appendix~\ref{a:Fadell_Rabinowitz}), this means $\FR(\Lambda)=\infty$.

For each $k\geq 1$, the \emph{$k$-th spectral invariant} of the smooth convex body $B$ is defined by 
\begin{align*}
 s_k(B) 
 := 
 \inf\big\{c>0\ \big|\ \FR(\{\Psi<c\})\geq k\big\},
\end{align*}
where $\{\Psi<\tau\}$ denotes the sublevel set $\Psi^{-1}(0,\tau)$. 
We recall some classical properties of these values:
\begin{itemize}
\setlength\itemsep{5pt}

\item \textbf{(Spectrality)} Every $s_k(B)$ is a critical value of the Clarke action functional associated with $B$, that is, an element of the action spectrum $\sigma(\partial B)$.

\item \textbf{(Systole)} $s_1(B)=\min\Psi=\sigma(\partial B)$.

\item \textbf{(Lusternik-Schnirelmann)} For all $k\geq1$ we have $s_k(B)\leq s_{k+1}(B)$. If $s_k(B)=s_{k+j}(B)$ for some $j\geq1$, then $\FR(\UU)>j$ for any $S^1$-invariant neighborhood $\UU\subset\Lambda$ of $\crit(\Psi)\cap\Psi^{-1}(s_k(B))$.

\item \textbf{($\bm{C^0}$ continuity)} For any sequence of smooth convex bodies $B_j$ converging to a smooth convex body $B$ in the $C^0$ topology, we have $s_k(B_j)\to s_k(B)$.

\end{itemize}

\subsection{Finite dimensional reduction}\label{ss:reduction}

The $C^{1,1}$ regularity of the Clarke action functional $\Psi$ is not sufficient to apply those results of Morse theory involving the Hessian, such as the Morse lemma. Nevertheless, $\Psi$ becomes smooth as soon as one restricts it to the subspace of $W^{1,2}$ curves that do not go through the origin of $\C^n$. With such restriction, however, one looses compactness properties such as the Palais-Smale condition. In order to both improve the regularity of $\Psi$ and retain the compactness properties, we will apply the finite dimensional reduction introduced by Ekeland and Hofer, which we now briefly recall. We refer the reader to \cite[Sec.~II.2]{Ekeland:1987aa} for the proofs.

Let $b>\min\Psi$ be a fixed value, and let us focus on the sublevel set $\{\Psi<b\}$ of the Clarke action functional. For each integer $N\geq1$, let $F=F_N\subset L^2_0(S^1,\C^n)$ be the $S^1$-invariant finite dimensional Hilbert subspace given by those $u\in L^2_0(S^1,\C^n)$ of the form
\begin{align*}
 u(t)=\sum_{0<|k|\leq N} e^{i2\pi kt}u_k,
\end{align*}
where $u_k\in\C^n$. We consider the open subset 
$V=\{u\in F\ |\ W_u\neq\varnothing\}$, where
\begin{align*}
 W_u:=\big\{ v\in F^\bot\ \big|\ \A(u+v)>0,\  \widetilde\Psi(u+v)<b \big\}.
\end{align*}
The function $W_u\to(0,\infty)$, $v\mapsto\widetilde\Psi(u+v)$ has a unique non-degenerate global minimizer $\nu(u)\in F^\bot$. We denote by $S\subset F$ the unit sphere of the Hilbert space $F$, and set $U:=V\cap S$. The \emph{reduced Clarke action functional} is defined as
\begin{align*}
 \psi:U\to(0,b),\qquad\psi(u)=\widetilde\Psi(u+\nu(u))=\Psi(\iota(u)),
\end{align*}
where 
\begin{align*}
\iota:U\hookrightarrow\Lambda,\qquad \iota(u):=\frac{u+\nu(u)}{\|u+\nu(u)\|_{L^2}}.
\end{align*}
The properties of this finite dimensional setting are the following.

\begin{itemize}
\setlength\itemsep{5pt}

\item \textbf{($\bm{S^1}$-invariance)} The map $\nu$ is $\C_*$-equivariant, and therefore the  reduced Clarke action functional $\psi$ is $S^1$-invariant.

\item \textbf{(Critical points)} Both $\nu$ and $\psi$ are $C^{1,1}$, and there is a one-to-one correspondence
\begin{align*}
\crit(\psi)&\ \bijection^{1:1}\ \crit(\Psi)\cap\Psi^{-1}(0,b)\\
u&\ \bijection\  
\iota(u).
\end{align*}
Both $\nu$, $\psi$, and the $S^1$ action are smooth on a sufficiently small neighborhood of the critical set $\crit(\psi)$.
Moreover, at every $u\in\crit(\psi)$, the Hessians $d^2\psi(u)$ and $d^2\Psi(\iota(u))$ have the same Morse index and the same nullity.

\item \textbf{(Compactness)} 
For each $c\in(0,b)$, the closed sublevel set  $\{\psi\leq c\}$ is compact.

\item \textbf{(Approximation)} For each $a\in(0,b)$, the map $\iota:\{\psi<a\}\hookrightarrow\{\Psi<a\}$ is an $S^1$-equivariant homotopy equivalence.

\end{itemize}

\subsection{Besse convex bodies}\label{ss:Besse}

We say that a $2n$-dimensional smooth convex body $B$ is \emph{Besse} when its boundary $\partial B$, equipped with the canonical contact form $\lambda|_{\partial B}$ of Equation \eqref{e:Liouville}, is a Besse contact manifold: all its Reeb orbits are closed, and therefore have a common period according to a theorem of Wadsley \cite{Wadsley:1975aa}. A theorem of Ginzburg, G\"urel, and Maz\-zuc\-chel\-li \cite{Ginzburg:2021aa} implies that the Besse property is detected by the spectral invariants: 

\begin{itemize}
 \item \textbf{(Spectral characterization)} We have $c:=s_k(B)=s_{k+n-1}(B)$ for some $k\geq1$ if and only if $B$ is Besse, and $c$ is a common period for the Reeb orbits on $\partial B$ (i.e.~$\phi_h^c=\id$, where $h:\C^n\to[0,\infty)$ is the 2-homogeneous Hamiltonian such that $h^{-1}(1)=\partial B$). In this case, the critical manifold $\crit(\Psi)\cap\Psi^{-1}(c)$ has Morse index $2k-2$ and nullity $2n-2$. Moreover, $c<s_{k+n}(B)$ and, when $k\geq2$, $c>s_{k-1}(B)$.
\end{itemize}

The Clarke action functional $\Psi$ associated with a Besse convex body $B$ also satisfies the following properties, which were established by Mazzucchelli and Radeschi \cite{Mazzucchelli:2023aa}. Analogous properties in the different setting of geodesic flows were established earlier by Radeschi and Wilkings \cite{Radeschi:2017aa}.

\begin{itemize}
\setlength\itemsep{5pt}

\item \textbf{(Critical manifolds)} For each $c\in\sigma(\partial B)$, the critical manifold 
\[\KK:=\crit(\Psi)\cap\Psi^{-1}(c)\cong \fix(\phi_h^c|_{\partial B})\] is a non-degenerate (i.e.\ $\ker d^2\Psi(u)=T_u\KK$ for all $u\in \KK$), connected, odd dimensional, rational homology sphere.

\item \textbf{(Perfectness)} The Clarke action functional $\Psi$ is perfect for the $S^1$-equi\-variant Morse theory with rational coefficients: for each $a<b$ the inclusions induce short exact sequences
\begin{align*}
\qquad\qquad
 0
 \to
 H^{*}_{S^1}(\{\Psi<b\},\{\Psi<a\})
 \to 
 H^{*}_{S^1}(\{\Psi<b\})
 \to 
 H^{*}_{S^1}(\{\Psi<a\})
 \to 
 0.
\end{align*}

\end{itemize}

\begin{Remark}
\label{r:s_k_ellipsoids}
For a rational ellipsoid $E$ (which is a Besse convex body) the perfectness of the associated Clarke action functional, together with the Lusternik-Schnirelmann property, readily implies that $s_k(E)=\cEH k(E)$ for all $k\geq 1$. 
Since $\aaa\mapsto s_k(E(\aaa))$ and $\aaa\mapsto \cEH k(E(\aaa))$ are continuous, the latter identity holds for irrational ellipsoids as well.
\hfill\qed
\end{Remark}

In order to prove our main Theorem~\ref{mt:local_max}, we shall prove analogous statements for the spectral invariants $s_k(B)$: Theorems \ref{t:main_Clarke} and, ultimately, \ref{t:final_Clarke}. A crucial ingredient for these statements is the following perturbation result for the spectral invariants. We recall that two $2n$-dimensional smooth convex bodies $B_0$ and $B_1$ are $C^k$-close if their boundaries are $C^k$-close embedded hypersurfaces of $\C^n$. 

\begin{Lemma}
\label{l:perturbation}
Let $B_0$ be a $2n$-dimensional Besse convex body, $k\geq1$ an integer such that $s_k(B_0)=s_{k+n-1}(B_0)$, and $[a,b]$ a compact neighborhood of $s_k(B_0)$ such that $[a,b]\cap\sigma(\partial B_0)=s_k(B_0)$. Then, for any smooth convex body $B_1$ that is sufficiently $C^2$-close to $B_0$, we have
\begin{align*}
 s_k(B_1)=\min \big(\sigma(\partial B_1)\cap[a,b]\big).
\end{align*}
\end{Lemma}

\begin{proof}[Proof of Lemma~\ref{l:perturbation}]
The statement is straightforward when $k=1$. Indeed, $B\mapsto s_k(B)$ is a continuous function with respect to the $C^0$ topology on the space of smooth convex bodies, and $s_1(B)=\min\sigma(\partial B)$ for each smooth convex body $B$. From here on, we shall assume $k\geq2$.

Let $h_0:\C^n\to[0,\infty)$ be the 2-homogeneous Hamiltonian such that $h_0^{-1}(1)=\partial B_0$, and 
$\Psi_0:\Lambda\to(0,\infty)$ the corresponding Clarke action functional. Consider the critical value 
$c:=s_k(B_0)=s_{k+n-1}(B_0)$. The spectral Besse characterization implies that the critical manifold $\KK:=\crit(\Psi_0)\cap\Psi_0^{-1}(c)$ is diffeomorphic to $\partial B_0$, non-degenerate, and with Morse index $d:=2k-2$.

Let $[a,b]$ be a compact neighborhood of $c$ such that $[a,b]\cap\sigma(\partial B_0)=c$, and consider the reduced Clarke action functional $\psi_0=\Psi_0\circ\iota_0:U_0\to(0,\infty)$, where the reduction is chosen so that $\psi_0$ approximates $\Psi_0$ up to a level larger than $b$  (see Section~\ref{ss:reduction}). We denote by $K:=\iota_0^{-1}(\KK)=\crit(\psi_0)\cap\psi_0^{-1}(c)$ the corresponding critical manifold of $\psi_0$, and recall that $K$ and $\KK$ have the same Morse index $d$. The normal bundle $N\subset TU_0|_K$ of $K$ admits an $S^1$-invariant splitting $N= N^+\oplus N^-$ such that $\pm d^2\psi_0(u)$ is positive definite on $N_u^\pm$ for each $u\in K$. The rank of $N^-$ is the Morse index $d$.
For each $u\in K$ and $r>0$, we denote by $D^\pm_u(r)\subset N^\pm_u$ the closed ball of radius $r$ centered at the origin, measured with the Riemannian metric of $U_0$. The union $D^\pm (r):=\cup_{u\in K}D^\pm_u(r)$ is the closed neighborhood of radius $r$ of the zero-section. By the Morse-Bott lemma, there exists a closed $S^1$-invariant neighborhood $D\subset U_0$ of $K$ small enough so that $\psi_0|_{D}$ is smooth, and an $S^1$-equivariant diffeomorphism identifying $D\equiv D^+(r_+)\oplus D^-(r_-)$, for some $r_->r_+>0$, so that $K$ is identified with the zero-section and
\[\psi_0|_D(u;x,y)=c+\|x\|^2-\|y\|^2,\qquad\forall u\in K,\ x\in D^+_u(r_+),\ y\in D^-_u(r_-).\]
In these local coordinates, we have $\partial^2_{yy}\psi_0|_D\equiv-I$. We require the radius $r_+$ to be small enough so that 
\begin{align*}
\max_{D}\psi_0=c+r_+^2<b.
\end{align*}
From now on, in order to simplify the notation, we simply write $D^\pm=D^\pm(r_\pm)$.

Since $B$ is Besse, $\Psi_0$ is perfect for the $S^1$-equivariant Morse theory with rational coefficients, and so is $\psi_0$ thanks to its approximation property (see Section~\ref{ss:reduction}). Therefore, the embedding $\iota_0:U_0\to\Lambda$ induces an isomomorphism
\begin{align*}
\iota_0^*: H^{d}_{S^1}(\Lambda,\{\Psi_0<c\})
 \ttoup^{\cong}
 H^{d}_{S^1}(D^-,\partial D^-).
\end{align*}
Since $K$ is simply connected, the negative normal bundle $N^-\to K$ is orientable, and therefore the inclusion and Thom isomorphism give
\begin{align*}
 H^{d}_{S^1}(D^-,\partial D^-)
 \ttoup^{\cong}
 H^{d}_{S^1}(D^-|_{S^1\cdot u},\partial D^-|_{S^1\cdot u})
 \cong 
 H^{0}_{S^1}(S^1\cdot u)
 \cong
 \Q,
 \\
 \forall u\in K.
\end{align*}
We consider the product $S^1\times D^-_u$ equipped with the $S^1$-action inherited from the first factor. For each $u\in K$ and $x\in D^+_u$, we consider the $S^1$-equivariant map
\begin{align*}
 j_{x}:S^1\times D^-_u\to D,
 \qquad
 j_{x}(t,y)=t\cdot(u;x,y),
\end{align*}
which depends continuously on $x$.
For $x=0$, the map $j_{0}$ is an $m$-fold covering map onto $D^-|_{S^1\cdot u}$, where $m\geq1$ is the order of iteration of $u$, i.e.~the largest integer $m$ such that $\tfrac1m \cdot u=u$. 
Moreover, $\psi_0\circ j_x|_{S^1\circ\partial D_u^-}\leq c+r_+^2-r_-^2<c$.
Therefore, the compositions $\iota_{0,x}:=\iota_0\circ j_{x}$ induce the isomomorphisms
\begin{align}
\label{e:iota_u_x_isom}
 \iota_{0,x}^*=
 \iota_{0,0}^*:
 H^{d}_{S^1}(\Lambda,\{\Psi_0<c\})
 \ttoup^{\cong}
 H^{d}_{S^1}(S^1\times D^-_u,S^1\times \partial D^-_u).
\end{align}

We fix $\epsilon>0$ small enough so that 
\begin{itemize}
 \setlength\itemsep{5pt}
\item[(i)] $[(1-\epsilon)^3c,(1-\epsilon)^{-2}c]\subset[a,b]$.

\item[(ii)] $c+r_+^2-r_-^2<(1-\epsilon)^2c$.
\end{itemize}
We consider a $2$-homogeneous Hamiltonian $h_1:\C^n\to[0,\infty)$ such that
\[\delta:=\|h_1-h_0\|_{C^2(B_0\setminus\{0\})}\]
is sufficiently small. Namely, $B_1:=h_1^{-1}[0,1]$ is a convex body $C^2$-close to $B_0$. For each $s\in[0,1]$,  the convex combinations $h_s:=sh_1+(1-s)h_0$ define the convex bodies $B_s:=h_s^{-1}[0,1]$ interpolating between $B_0$ and $B_1$. 
We require $\delta\leq\epsilon$, so that
\begin{align}
\label{e:sandwich_h0_h1}
 (1-\epsilon)h_s \leq h_0 \leq (1-\epsilon)^{-1}h_s,\qquad\forall s\in[0,1],
\end{align}
We denote by $\Psi_s:\Lambda\to(0,\infty)$ the Clarke action functional associated with $h_s$, and $\psi_s=\Psi_s\circ\iota_s:U_s\to(0,\infty)$ its finite dimensional reduction. Up to choosing $\delta$ small enough, we can insure that all the domains $U_s$ are open subsets of the same finite dimensional vector subspace of $L^2_0(S^1;\C^n)$ and contain the compact set $D$.
The dual Hamiltonians $h_s^*$ satisfy  inequalities analogous to~\eqref{e:sandwich_h0_h1}, and so do the  Clarke action functionals $\Psi_s$, i.e. \begin{align}
\label{e:Psi_01}
 (1-\epsilon)\Psi_s \leq \Psi_0 \leq (1-\epsilon)^{-1}\Psi_s,\qquad\forall s\in[0,1].
\end{align}
We require $\delta$ to be small enough so that, for each $s\in[0,1]$, we have:
\begin{itemize}
 \setlength\itemsep{5pt}

 \item[(iii)] $\sigma(\partial B_s)\cap[a,b]=\sigma(\partial B_s)\cap[(1-\epsilon)c,(1-\epsilon)^{-1}c]$;

 \item[(iv)] $\psi_s|_D$ is smooth and $C^2$-close to $\psi_0$, and in particular $\psi_s|_D\leq(1-\epsilon)^{-1}\psi_0|_D$, and $\partial^2_{yy}\psi_s$ is negative definite at all points of $D$. 
 
 \item[(v)] $\crit(\psi_s)\cap\psi_s^{-1}[a,b]\subset D$.
 
\end{itemize}

The inequalities~\eqref{e:Psi_01} imply that we have the inclusions of sublevel sets
\begin{align}
\label{e:sublevel_inclusions}
\big\{ \Psi_s < (1-\epsilon)^3 c \big\}
\hookrightarrow
\big\{ \Psi_0 < (1-\epsilon)^2 c \big\}
\hookrightarrow
\big\{ \Psi_s < (1-\epsilon)c \big\}
\hookrightarrow
\big\{ \Psi_0 < c \big\}.
\end{align}
Notice that $\big[(1-\epsilon)^3c,(1-\epsilon)c\big)$ is an interval of regular values of $\Psi_s$, and $\big[(1-\epsilon)^2c,c\big)$ is an interval of regular values of $\Psi_0$, according to properties (i) and (iii). Therefore, the composition of any two subsequent inclusions in~\eqref{e:sublevel_inclusions} is an $S^1$-equivariant homotopy equivalence. This implies that each of the inclusions in~\eqref{e:sublevel_inclusions} induces an isomorphism in $S^1$-equivariant cohomology, in particular the last one:
\begin{align}
\label{e:Psi_01_iso}
 H^*_{S^1}(\{ \Psi_0 < c \})\ttoup^{\cong}H^*_{S^1}(\{ \Psi_s < (1-\epsilon)c \}).
\end{align}
The inequalities~\eqref{e:Psi_01} and property (i) further imply that we have inclusions
\begin{align*}
 \{\Psi_0\leq c\}\hookrightarrow
 \{\Psi_s\leq (1-\epsilon)^{-1}c\}\hookrightarrow
 \{\Psi_0\leq b\},
\end{align*}
and since their composition is an $S^1$-equivariant homotopy equivalence, the second inclusion induces an injective homomorphism
\begin{align}
\label{e:Psi_01_injective}
 H^*_{S^1}(\{\Psi_0\leq b\})\hookrightarrow
  H^*_{S^1}(\{\Psi_s\leq (1-\epsilon)^{-1}c\}).
\end{align}
Equations~\eqref{e:Psi_01_iso} and \eqref{e:Psi_01_injective} imply that 
\[
s_k(\partial B_1)
\in
\big[(1-\epsilon)c,(1-\epsilon)^{-1}c\big).
\]
In particular, $\big[(1-\epsilon)c,(1-\epsilon)^{-1}c\big)$ contains some critical values of the Clarke action functional $\Psi_1$, and we denote the smallest one by
\[c_1:=\min\big(\sigma(\partial B_1)\cap[a,b]\big)\in[(1-\epsilon)c,(1-\epsilon)^{-1}c].\]

For each $s\in[0,1]$ and $x\in D_u^+$, properties (ii) and (iv) imply
\begin{align*}
 \psi_s\circ j_x|_{S^1\times\partial D_u^-}
 &\leq
 (1-\epsilon)^{-1}\psi_0\circ j_x|_{S^1\times\partial D_u^-}\\
 &\leq
 (1-\epsilon)^{-1}(c+r_+^2-r_-^2)\\
 &<
 (1-\epsilon)c.
\end{align*}
Therefore, if we denote $\iota_{s,x}:=\iota_s\circ j_x$, we have
\begin{align*}
 \iota_{s,x}(S^1\times \partial D_u^-)
 \subset
 \{\Psi_s<(1-\epsilon)c\}
 \subset
 \{\Psi_0<c\}.
\end{align*}
Since $\iota_{s,x}$ depends continuously on $s$ and $x$, it induces the same isomorphism as in \eqref{e:iota_u_x_isom} in relative cohomology, i.e.
\begin{align}
\label{e:iota_u_x_isom_s}
 \iota_{s,x}^*=
 \iota_{0,0}^*:
 H^{d}_{S^1}(\Lambda,\{\Psi_0<c\})
 \ttoup^{\cong}
 H^{d}_{S^1}(S^1\times D^-_u,S^1\times \partial D^-_u).
\end{align}

We consider a critical circle $S^1\cdot p\in\crit(\psi_1)\cap\psi_1^{-1}(c_1)$. Property (v) guarantees that $S^1\cdot p\subset D$, and therefore we can write $p$ as 
\[p=(u;x,y)\in D^+\oplus D^-\equiv D.\]
Property (iv) implies that the restriction $\psi_1|_{\{x\}\oplus D_u^-}$ is a strictly concave function with a unique maximizer at $y$. Therefore 
\[\iota_{1,x}(S^1\times D_u^-)\subset\{\Psi_1\leq c_1\}.\] 
Therefore, the isomomorphism~\eqref{e:iota_u_x_isom_s} factors as
\[
\begin{tikzcd}[row sep=large]
H^{d}_{S^1}(\Lambda,\{\Psi_0<c\})
\arrow[r,"\iota_{s,x}^*","\cong"']
\arrow[d,"i^*"']
&
 H^{d}_{S^1}(S^1\times D^-_u,S^1\times \partial D^-_u)
\\
H^{d}_{S^1}(\{\Psi_1\leq c_1\},\{\Psi_1\leq (1-\epsilon)c\})
\arrow[ur,"\iota_{s,x}^*"']
\end{tikzcd}
\]
where $i^*$ is induced by the inclusion. The diagram implies that $i^*$ is injective. Finally, consider the following commutative diagram, whose rows are  parts of long exact sequences of inclusions, and whose vertical homomorphisms are induced by inclusions:
\[
\small
\begin{tikzcd}[row sep=large]
H^{d-1}_{S^1}(\{\Psi_0<c\})
\arrow[r,"\partial^*=0"]
\arrow[d,"\cong"']
&
H^{d}_{S^1}(\Lambda,\{\Psi_0<c\})
\arrow[r,"l^*",hookrightarrow]
\arrow[d,"i^*",hookrightarrow]
&
H^{d}_{S^1}(\Lambda)
\arrow[d,"j^*"]
\\
H^{d-1}_{S^1}(\{\Psi_1<(1-\epsilon)c\})
\arrow[r,"\partial^*"]
&
H^{d}_{S^1}(\{\Psi_1\leq c_1\},\{\Psi_1<(1-\epsilon)c\})
\arrow[r,"m^*"]
&
H^{d}_{S^1}(\{\Psi_1\leq c_1\})
\end{tikzcd}
\]
The cohomology class $e^{k-1}\in H^d_{S^1}(\Lambda)$ belongs to the image of $l^*$, since $c=s_k(\partial B_0)$. Let $f\in H^{d}_{S^1}(\Lambda,\{\Psi_0<c\})$ be the non-zero cohomology class such that $l^*f=e^{k-1}$. The diagram readily implies that the connecting homomorphism $\partial^*$ on the bottom line is the zero homomorphism, namely  that $m^*$ is injective. Therefore,  \[j^*e^{k-1}=m^*i^*f\neq0,\] 
which implies $s_k(\partial B_1)\leq c_1$. Since $s_k(\partial B_1)\in\big[(1-\epsilon)c,(1-\epsilon)^{-1}c\big)$, and since $c_1$ is the smaller critical value of $\Psi_1$ in $\big[(1-\epsilon)c,(1-\epsilon)^{-1}c\big)$, we conclude that $s_k(\partial B_1)= c_1$.
\end{proof}

\subsection{Local maximizers of the spectral ratio}\label{ss:local_max_sk}

From now on, we will focus on dimension 4. Let $B$ be a $4$-dimensional smooth convex body, $h:\C^2\to[0,\infty)$ the 2-homogeneous Hamiltonian such that $h^{-1}(1)=\partial B$, and $\Psi:\Lambda\to(0,\infty)$ the associated Clarke action functional. We denote by $\tau_k(B)$ the infimum of the values $\tau>0$ such that $\crit(\Psi)\cap\Psi^{-1}(0,\tau)$ contains at least $k$ critical circles. In terms of the Hamiltonian  flow $\phi_h^t$, this can be expressed as
\begin{align*}
 \tau_k(B):=\inf\bigg\{\tau>0\ \bigg|\ \sum_{t\in(0,\tau]} \#\Big(\tfrac{\fix(\phi_h^t)\setminus\{0\}}{\sim}\Big)\geq k\bigg\},
\end{align*}
where $\sim$ is the equivalence relation identifying $z\sim c\,\phi_h^t(z)$ for all $t\in\R$ and $c>0$.

Assume that $B$ is Besse. We denote by $\tau(B)$ the minimal common period of the closed Reeb orbit on $\partial B$, namely the minimal $t>0$ such that $\phi_h^t=\id$. The critical set $\crit(\Psi)\cap\Psi^{-1}(0,\tau(B))$ consists of a finite union of non-degenerate critical circles (indeed, since $B$ is symplectomorphic to an ellipsoid according to the forthcoming Proposition~\ref{p:Besse_ellipsoids}, at most two critical circles). More generally, for each integer $m\geq1$, $\crit(\Psi)\cap\Psi^{-1}(m\tau(B),(m+1)\tau(B))$ consists of a finite union of non-degenerate critical circles: the suitable iterates of the critical circles in $\crit(\Psi)\cap\Psi^{-1}(0,\tau(B))$.
Since the Clarke action functional $\Psi$ is perfect, every critical value of $\Psi$ is equal to some spectral invariant $s_k(B)$, i.e.
\begin{align*}
\sigma(\partial B)
=
\big\{s_k(B)\ \big|\ k\geq1\big\}.
\end{align*}
For all integers $m\geq1$, we define
\begin{align*}
 k_m=k_m(B):= \min\big\{ k\geq1\ \big|\  s_k(B)=m\,\tau(B)\big\}.
\end{align*}
If $k_1\geq2$, for each $k\in\{1,...,k_1-1\}$ the critical set $\crit(\Psi)\cap\Psi^{-1}(s_k(B))$ is a single non-degenerate critical circle. Therefore, the perfectness and the Lusternik-Schinerlmann property of $\Psi$ imply that
\begin{align}
\label{e:s_k=tau_k_final}
 s_k(B)=\tau_k(B),\qquad\forall k\in\{1,...,k_1\}.
\end{align}
Moreover, the spectral characterization of Besse convex bodies (see Section~\ref{ss:Besse}) implies that $s_{k_m}(B)=s_{k_m+1}(B)<s_{k_m+2}(B)$ and, if $k_m>1$, $s_{k_m-1}(B)<s_{k_m}(B)$.
We define 
\[
K(B):=\big\{k_m(B)\ \big|\ m\geq1  \big\}.
\]

\begin{Lemma}
\label{l:s_k=tau_k}
Let $B_0$ be a $4$-dimensional Besse convex body, and $m\geq1$ an integer. For each smooth convex body $B_1$ that is sufficiently $C^2$-close to $B_0$, we have
\begin{align*}
s_{k_m(B_0)}(B_1)\leq m\,\tau_{k_1(B_0)}(B_1).
\end{align*}
\end{Lemma}

\begin{proof}
Fix an integer $m\geq1$, and consider $k_1=k_1(B_0)$ and $k_m=k_m(B_0)$.  Let $[a,b]\subset(0,\infty)$ be a compact neighborhood of $s_{k_m}(B_0)$ that is small enough so that \[\sigma(\partial B_0)\cap[a,b]=\{s_{k_m}(B_0)\}.\]
Let $\epsilon>0$ be small enough so that $m(s_{k_1}(B_0)-\epsilon)\geq a$. Let $\Psi$ be the Clarke action functional associated with $B_0$. Since $\crit(\Psi)\cap\Psi^{-1}(\tau_{k}(B_0))$ is a non-degenerate critical circle for each $k\in\{1,...,k_1-1\}$, 
any smooth convex body $B_1$ that is sufficiently $C^2$-close to $B_0$ satisfies
\begin{align*}
 |\tau_k(B_1)-\tau_k(B_0)|<\epsilon,\qquad\forall k\in\{1,...,k_1\}.
\end{align*}
This, together with \eqref{e:s_k=tau_k_final}, implies $\tau_{k_1}(B_1)>s_{k_1}(B_0)-\epsilon$, and therefore \[m\,\tau_{k_1}(B_1)\geq a.\] 
By Lemma~\ref{l:perturbation}, we have 
\[s_{k_m}(B_1)=\min\big(\sigma(\partial B_1)\cap[a,b]\big).\]
Since $m\,\tau_{k_1}(B_1)\in\sigma(\partial B_1)$, we conclude $s_{k_m}(B_1)\leq m\,\tau_{k_1}(B_1)$.
\end{proof}

We can finally prove the main result of this section. For each $k\geq1$ and for each $4$-dimensional smooth convex body $B\subset\C^n$, we set
\begin{align*}
 \widehat\tau_k(B)=\frac{\tau_k(B)}{\vol(B)^{1/2}},
 \qquad
 \widehat s_k(B)=\frac{s_k(B)}{\vol(B)^{1/2}},
\end{align*}
where the volume $\vol(B)=\vol(B,\omega)$ is obtained by integrating $\omega\wedge\omega$.

\begin{Thm}\label{t:main_Clarke}
Let $B_0$ be a $4$-dimensional Besse convex body. For each integer $k\in K(B_0)$, any smooth convex body $B_1$ that is sufficiently $C^3$-close to $B_0$ satisfies $\widehat{s}_k(B_1)\leq \widehat{s}_k(B_0)$.
\end{Thm}

\begin{proof}
Consider the integers $k_m=k_m(B_0)$ associated with $B_0$. The spectral invariants $s_{k_m}(B_0)$ are given by
\begin{align*}
 s_{k_m}(B_0)=m \,\tau_{k_1}(B_0).
\end{align*}
A theorem of Abbondandolo, Lange, and Mazzucchelli \cite{Abbondandolo:2022aa}, extending an earlier theorem due to Abbondandolo, Bramham, Hryniewicz, and Salomao \cite{Abbondandolo:2018aa} for the special case $k_1=1$, implies that any smooth convex body $B_1$ that is sufficiently $C^3$-close to $B_0$ satisfies
\begin{align*} 
\widehat\tau_{k_1}(B_1)\leq\widehat\tau_{k_1}(B_0).
\end{align*}
Fix an integer $m\geq1$. By Lemma~\ref{l:s_k=tau_k}, if $B_1$ is sufficiently $C^2$-close to $B_0$, we have
\begin{align*}
s_{k_m}(B_1)\leq m\,\tau_{k_1}(B_1). 
\end{align*}
Overall, we proved that $\widehat s_{k_m}(B_1)\leq \widehat s_{k_m}(B_0)$ for each smooth convex body $B_1$ that is sufficiently $C^3$-close to $B_0$.
\end{proof}

In Theorem~\ref{t:final_Clarke} at the end of Section~\ref{ss:local_max_chat_k}, we will strengthen this theorem and fully characterize the local maximizers of the spectral ratios $\widehat s_k$.

\section{The Ekeland-Hofer capacities}
\label{s:EH}

\subsection{Definition of the capacities}
\label{ss:def_EH_capacities}
We begin by briefly recalling Ekeland and Hofer's construction \cite{Ekeland:1989aa, Ekeland:1990aa} of their symplectic capacities. The functional setting  involves the Sobolev space $E:=W^{1/2,2}(S^1,\C^n)$, which is the space of all periodic curves $\gamma:S^1\to\C^n$ of the form
\begin{align*}
\gamma(t)=\sum_{k\in\Z} e^{i2\pi kt}\gamma_k, 
\end{align*}
where $\gamma_k\in\C^n$, such that 
\[ 
\|\gamma\|_{E}^2:= \sum_{k\in\Z}|k|\cdot\|\gamma_k\|^2<\infty.
\]
The circle $S^1$ acts on $E$ by translations, i.e. $t\cdot\gamma=\gamma(t+\cdot)$ for all $t\in S^1$ and $\gamma\in E$.
The Sobolev space $E$ admits an orthogonal direct sum decomposition 
\begin{align}
\label{e:E=E-_E0_E+}
E=E_-\oplus E_0\oplus E_+,
\end{align} 
where $E_\pm$ consists of the $\gamma$ such that $\gamma_k=0$ for all $\mp k\geq0$, and $E_0$ consists of the constants $\gamma\equiv \gamma_0$. We consider the quadratic form $a:E\to\R$ given by
\begin{align*}
 a(\gamma)=\pi\sum_{k\in\Z} k\,\|\gamma_k\|^2.
\end{align*}
On the dense subspace $W^{1,2}(S^1,\C^n)\subset E$, this quadratic form is related to the quadratic form $\AAA:L^2_0(S^1,\C^n)\to\R$ employed in the previous section by \[a(\gamma)=\AAA(\dot\gamma).\] 

The Hamiltonians involved in Ekeland-Hofer's setting are non-negative smooth functions of the form $H:\C^n\to[0,\infty)$ that vanish on some non-empty open subset of $\C^n$ and have the form $H(z)=r\|z\|^2$ outside a compact set, for some $r>\pi$ that is not a multiple of $\pi$; we briefly refer to such $H$ as to \emph{admissible} Hamiltonians. The associated Hamiltonian action functional $\Phi_H:E\to\R$, given by
\begin{align*}
 \Phi_H(\gamma)=a(\gamma)-\int_0^1 H(\gamma(t))\,dt,
\end{align*}
is $C^\infty$, $S^1$-invariant, satisfies the Palais-Smale condition, and its critical circles $S^1\cdot\gamma$ are precisely the circles of curves of the form $\gamma(t)=\phi_H^t(\gamma(0))$, where $\phi_H^t$ is the Hamiltonian flow of $H$. However, unlike for the Clarke action functional, the critical circles of $\Phi_H$ have infinite Morse index.

A preliminary step to construct the Ekeland-Hofer capacities consists in defining suitable spectral invariants for $\Phi_H$. Since all Morse indices are infinite, the recipe is not an ordinary min-max, but a more sophisticated linking argument. We write each element of $E$ as $\gamma=\gamma_-+\gamma_0+\gamma_+$ according to the orthogonal decomposition~\eqref{e:E=E-_E0_E+}.
Let $\DD$ be the group of those $S^1$-equivariant homeomorphisms $\Theta:E\to E$ of the form
\begin{align}
\label{e:Theta}
 \Theta(\gamma) 
 = 
 e^{\theta_-(\gamma)}\gamma_-
 +
 \gamma_0
 +
 e^{\theta_+(\gamma)}\gamma_+ + \KK(\gamma),
\end{align}
where $\theta_\pm:E\to\R$ and $\KK:E\to E$ are continuous $S^1$-invariant maps that send bounded sets to precompact sets, vanish on the sublevel set $\{a\leq0\}$, and vanish as well outside a ball of sufficiently large radius in $E$. Let $S_+$ be the unit-sphere of the Hilbert subspace $E_+\subset E$. The \emph{Ekeland-Hofer index} of an $S^1$-invariant subset $\VV\subset E$ is defined as
\begin{align*}
 \EH(\VV)=\inf_{\Theta\in\DD}\ \inf_\UU\ \FR(\UU),
\end{align*}
where the inner infimum ranges over all $S^1$-invariant open subsets $\UU\subset E$ containing $\VV\cap\Theta(S_+)$. For each integer $k\geq1$, we define the $k$-th spectral invariant associated with the Ekeland-Hofer index as
\begin{align*}
 \cEH k(H) := \inf\big\{ c\in\R\ \big|\ \EH(\{\Phi_H<c\})\geq k\big\}.
\end{align*}
These values satisfy the following properties:
\begin{itemize}
\setlength\itemsep{5pt}

\item \textbf{(Finiteness)} We have $\cEH1(H)\leq \cEH2(H)\leq...\leq \cEH{nk}(H)<\infty$ for all Hamiltonians $H\in\FF(B)$ such that $H(z)>k\pi\|z\|^2$ outside a compact set. From now on, whenever we write $\cEH k(H)$ we implicitly assume that $H$ satisfies this latter condition.

\item \textbf{(Positivity)} $\cEH1(H)>0$.

\item \textbf{(Spectrality)} Every $\cEH k(H)$ is a critical value of the Hamiltonian action functional $\Phi_H$.

\item \textbf{(Monotonicity)} If $H_1\leq H_2$ pointwise, then $\cEH k(H_1)\geq \cEH k(H_2)$.

\end{itemize}

For a bounded subset $B\subset\C^n$, we denote by $\FF(B)$ the family of admissible Hamiltonians $H:\C^n\to[0,\infty)$ vanishing on some neighborhood of the closure $\overline B$. The \emph{$k$-th Ekeland-Hofer capacity} of $B$ is defined as 
\begin{align*}
\cEH k(B) = \inf_{H\in\FF(B)} \cEH k(H).
\end{align*}
For an unbounded subset $U\subset\C^n$, the $k$-th Ekeland-Hofer capacity $\cEH k(U)$ is defined as the supremum of $\cEH k(B)$ over all bounded subsets $B\subset U$. The $\cEH k$'s satisfy the monotonicity, conformality, and normalization properties of symplectic $k$-capacities (see Section~\ref{ss:result}).

Throughout this section, in order to simplify the notation, we shall denote the Ekeland-Hofer spectral invariants and capacities $\cEH k$ simply by $c_k$.

\subsection{Capacities vs spectral invariants}

We say that a closed hypersurface $N\subset\C^n$ is of \emph{restricted contact type} when there exists a primitive $\nu$ of the standard symplectic form $\omega$ of $\C^n$ such that $\nu|_{N}$ is a contact form. Notice that the action spectrum $\sigma(N)$, which is the collection of the periods of the closed Reeb orbits of $(N,\nu)$, is independent of the specific choice of the primitive $\nu$. Ekeland and Hofer \cite[Prop.~2]{Ekeland:1990aa} established the following remarkable facts.

\begin{itemize}
\setlength\itemsep{5pt}

\item \textbf{(Hausdorff continuity)} The functions $B\mapsto c_k(B)$ are continuous over the space of connected compact domains $B\subset\C^n$ with smooth restricted contact type boundary, endowed with the topology induced by the Hausdorff distance.
 
\item \textbf{(Spectrality)} For each connected compact subset $B\subset\C^n$ with smooth restricted contact type boundary, we have $c_k(B)\in\sigma(\partial B)$.

\end{itemize}

Notice in particular that, if $B\subset\C^n$ is a smooth convex body, both the Ekeland-Hofer capacities $c_k(B)$ and the spectral invariants $s_k(B)$ (see Section~\ref{ss:Clarke_spectral_invariants}) are elements of the action spectrum $\sigma(\partial B)$. A result of Sikorav \cite[Section 6.5]{Sikorav:1990aa} shows that these invariants coincide for $k=1$, i.e.
\begin{align}
\label{e:c1_systole}
c_1(B)=s_1(B)=\min\sigma(\partial B). 
\end{align}
It is not known whether $c_k(B)=s_k(B)$ for $k\geq2$ as well. In this section, we prove that there is at least an inequality.

\begin{Prop}
\label{p:capacity<spectral}
For each smooth convex body $B\subset\C^n$, we have 
\[c_k(B)\leq s_k(B),\qquad\forall k\geq2.\]
\end{Prop}

The proof of this proposition is technical, and will take the rest of this subsection. A reader mainly interested in the application may skip it and proceed directly to the next Section~\ref{ss:local_max_chat_k}.

Let $B\subset\C^n$ be a smooth convex body, $h:\C^n\to[0,\infty)$ be the 2-homogeneous Hamiltonian such that $h^{-1}(1)=\partial B$,  $\widetilde\Psi=\HH/\AAA:\AAA^{-1}(0,\infty)\to(0,\infty)$ the associated (unrestricted) Clarke action functional, and $\Psi=\widetilde\Psi|_\Lambda$ its restriction. The following lemma is  extracted from the arguments in \cite[Sec.~6.5]{Sikorav:1990aa}.

\begin{Lemma}
\label{l:Sikorav}
For each $\gamma\in W^{1,2}(S^1,\C^n)\cap a^{-1}(0,\infty)$ and $\zeta\in E_-\oplus E_0\oplus\R\gamma$, we have
\begin{align*}
a(\zeta)\leq \widetilde\Psi(\dot\gamma)\int_{S^1} h(\zeta(t))\,dt.
\end{align*}
\end{Lemma}

\begin{proof}
Since $a(\gamma)>0$ and $a|_{E_-\oplus E_0}\leq0$, the polynomial 
\[
c\mapsto a(\zeta+c \gamma) = a(\zeta) +  c\, \langle J\zeta,\dot \gamma\rangle_{L^2} + c^2 a(\gamma)
\] 
has at least one root. Therefore its discriminant is non-negative, i.e.
\begin{align}
\label{e:Sikorav}
a(\gamma)a(\zeta)
\leq
\frac{\langle J\zeta,\dot \gamma\rangle_{L^2}^2}{4}
\end{align}
The dual Hamiltonian $h^*$, defined by~\eqref{e:dual_Hamiltonian}, satisfies the Fenchel inequality 
\[h(z)+h^*(w)\geq \langle w,z\rangle.\]
Therefore, for each $c>0$, we have
\begin{align*}
c^2\HH(\dot\gamma) = 
\HH(c\dot\gamma) = 
\int_{S^1} h^*(-J c\dot \gamma)\,dt \geq 
c\,\langle J\zeta,\dot\gamma\rangle_{L^2} - \int_{S^1} h(\zeta)\,dt,
\end{align*}
that is,
\begin{align*}
c\,\langle J\zeta,\dot\gamma\rangle_{L^2} - c^2\HH(\dot\gamma)
\leq
\int_{S^1} h(\zeta)\,dt.
\end{align*}
The left-hand side is maximal for $c=\tfrac12\langle J\zeta,\dot\gamma\rangle_{L^2}\HH(\dot\gamma)^{-1}$, and we obtain the estimate
\begin{align*}
 \frac{\langle J\zeta,\dot\gamma\rangle_{L^2}^2}{4}
\leq
\HH(\dot\gamma)
\int_{S^1} h(\zeta)\,dt.
\end{align*}
This, together with \eqref{e:Sikorav}, implies the lemma.
\end{proof}

The Hilbert spaces $L^2_0(S^1,\C)$ and $E$ admit the $S^1$ action by time-translations, and the $\R_+=(0,\infty)$ action by scalar multiplication. Overall, they thus admit an action of $\C_*\equiv\R_+\times S^1$. We introduce the $\C_*$-equivariant injective continuous map
\begin{align*}
 i:L^2_0(S^1,\C^n)\to  E_-\oplus E_+,
 \qquad
 i(\dot\gamma)=\gamma-\int_{S^1}\gamma(t)\,dt.
\end{align*}
For each $c>\min\sigma(\partial B)$ we consider the $c$-sublevel sets of the unrestricted and restricted Clarke action functionals, which are related by
\begin{align*}
 \{\widetilde\Psi\leq c\}=\R_+\{\Psi\leq c\}.
\end{align*}
The inclusion $\{\Psi\leq c\}\hookrightarrow\{\widetilde\Psi\leq c\}$ is an $S^1$-equivariant homotopy equivalence, and therefore
\[\FR(\{\widetilde\Psi\leq c\})=\FR(\{\Psi\leq c\}).\]
We set
\[
\UU_c:=i\big(\{\widetilde\Psi\leq c\}\big)
\subset
(E_-\oplus E_+)\setminus E_-.
\]

\begin{Lemma}\label{l:Sikorav2}
For all $\epsilon>0$ there exists an admissible Hamiltonian $H\in\FF(B)$ such that
\begin{align*}
 \sup_{E_-\oplus E_0\oplus \UU_c} \Phi_H\leq c+\epsilon.
\end{align*}
\end{Lemma}

\begin{proof}
Since $B=h^{-1}[0,1]$, we can choose an admissible Hamiltonian $H\in\FF(B)$ such that $H\geq c\,h-c-\epsilon$. For each $\zeta\in E_-\oplus E_0\oplus \UU_c$, Lemma~\ref{l:Sikorav} implies
\[
\Phi_H(\zeta) 
=
a(\zeta) - \int_{S^1} H(\zeta)\,dt
<
c \int_{S^1} h(\zeta)\,dt - \int_{S^1} H(\zeta)\,dt
\leq
c+\epsilon.
\qedhere
\]
\end{proof}

For any integer $N\geq0$, we introduce the finite dimensional Hilbert subspace $F_N\subset E$ of all periodic curves $\gamma\in E$ of the form
\begin{align}
\label{e:Fourier}
\gamma(t)=\sum_{|k|\leq N} e^{i2\pi kt}\gamma_k.
\end{align}
We denote by $\Pi_N:E\to F_N$ the orthogonal projection.

\begin{Lemma}
\label{l:projection_Pi_N}
For each $c>\min\sigma(\partial B)$ there exists $N_c\geq1$ such that
\begin{align*}
\Pi_N(\UU_c)\cap E_-=\varnothing,
\qquad
\forall N\geq N_c.
\end{align*}
\end{Lemma}

\begin{proof}
Consider an arbitrary $\gamma\in W^{1,2}(S^1,\C^n)$ such that $\gamma_0=0$, $a(\gamma)>0$ and $\Pi_N\gamma\in E_-$. This latter condition implies
\begin{align}
\label{e:a_upper}
 a(\gamma) 
 = 
 \underbrace{a(\Pi_N\gamma)}_{\leq0} + a((\id-\Pi_N)\gamma)
 \leq
 a((\id-\Pi_N)\gamma)
 \leq
 \pi \sum_{k>N} k\|\gamma_k\|^2.
\end{align}
Consider the constant
\begin{align*}
 r:= \inf_{w\in\C^n\setminus\{0\}} \frac{h^*(w)}{\|w\|^2}>0,
\end{align*}
which we employ to bound
\begin{align}
\label{e:H_lower}
 \HH(\dot\gamma) 
 \geq 
 r \|\dot\gamma\|_{L^2}^2
 \geq
 r
 \sum_{k>N}
 4\pi^2 k^2 \|\gamma_k\|^2
 \geq
 4\pi^2 r N
 \sum_{k>N}
 k \|\gamma_k\|^2.
\end{align}
The estimates \eqref{e:a_upper} and \eqref{e:H_lower} imply
\begin{align*}
\widetilde\Psi(\dot\gamma)
=
\frac{\HH(\dot\gamma)}{a(\gamma)}
\geq
4\pi r N,
\end{align*}
and the desired quantity of the lemma can be set to $N_c:=c/(4\pi r)$.
\end{proof}

For technical reasons, it will be useful to consider a compact $S^1$-invariant subset of $\UU_c$. 

\begin{Lemma}
\label{l:compactness_EH}
There exists an $S^1$-invariant compact subset $\VV_c\subset\UU_c$ such that
\begin{align*}
 \FR(\R_+\VV_c) \geq \FR(\VV_c) \geq \FR(\{\Psi\leq c\}).
\end{align*}
\end{Lemma}

\begin{proof}
Consider the reduced Clarke action functional $\psi:U\to(0,c+1)$ introduced in Section~\ref{ss:reduction}. We recall that $U$ is an open subset of the unit-sphere $S\subset F$, where $F=F_N$ is the finite dimensional Hilbert space introduced above, and $N$ is sufficiently large. The function $\psi$ has the form 
\[\psi(u)=\widetilde\Psi(u+\nu(u)),\] 
where $\nu:U\to L^2_0(S^1,\C^n)$ is an $S^1$-equivariant $C^{1,1}$ map. The sublevel set
\begin{align*}
 \KK_c:= \big\{ u+\nu(u)\ \big|\ \psi(u)\leq c \big\}
\end{align*}
is compact in $L^2_0(S^1,\C^n)$, and the inclusion $\KK_c\hookrightarrow\{\widetilde\Psi\leq c\}$ is an $S^1$-equivariant homotopy equivalence. In particular
\begin{align*}
 \FR(\KK_c)=\FR(\{\widetilde\Psi\leq c\})=\FR(\{\Psi\leq c\}).
\end{align*}
The image $\VV_c=i(\KK_c)$ is an $S^1$-invariant compact subset of $\UU_c$, and by the monotonicity property of the Fadell-Rabinowitz index (see Appendix~\ref{a:Fadell_Rabinowitz}) we conclude
\[
 \FR(\R_+\VV_c)\geq \FR(\VV_c)\geq \FR(\KK_c).
\qedhere
\]
\end{proof}

In \cite[Prop.~1]{Ekeland:1990aa}, Ekeland and Hofer proved that 
\[\EH(E_-\oplus E_0\oplus V)=\dim_\C(V)=\FR(V\setminus\{0\})\] for every finite dimensional $S^1$-invariant vector subspace $V\subset E_+$. Building on Ekeland and Hofer arguments, we need to prove the following analogous statement for the $\C_*$-invariant set $\WW_c:=\R_+\VV_c$. We shall employ the same notation as in Appendix~\ref{a:Fadell_Rabinowitz}: if $W$ is a subset of a vector space, we set $W_*:=W\setminus\{0\}$.

\begin{Lemma}
\label{l:E-_E0_X}
For each $c>\min\sigma(\partial B)$, we have
$\EH(E_-\oplus E_0\oplus \WW_c)\geq\FR(\WW_c)$.
\end{Lemma}

\begin{proof}
Fix an arbitrary $\Theta\in\DD$, and an arbitrary $S^1$-equivariant open neighborhood $\UU\subset E$ of $\Theta(E_-\oplus E_0\oplus\WW_c) \cap S_+$. We need to prove 
\begin{align*}
 \FR(\UU)\geq\FR(\WW_c).
\end{align*}
We consider the finite dimensional Hilbert subspace $F=F_N\subset E$ introduced above, for a sufficiently large $N$ that we will fix later. This Hilbert subspace splits as a direct sum $F = F_- \oplus F_0\oplus F_+$, where $F_0=E_0$ and
$F_\pm =F_{N,\pm} :=E_{\pm}\cap F$.
We denote by $\Pi_\pm=\Pi_{N,\pm}:E\to F_\pm$ and $\Pi=\Pi_N:E\to F$ the orthogonal projections.  
We require $N\geq N_c$, so that Lemma~\ref{l:projection_Pi_N} implies that $\Pi_+(\WW_c)$ does not contain the origin. We set
\[W:=\Pi_+(\WW_c)\cup\{0\},\qquad W_*:=\Pi_+(\WW_c)=W\setminus\{0\},\] so that
\begin{align*}
 \Pi(E_-\oplus E_0\oplus\WW_c)=F_-\oplus F_0\oplus W_*.
\end{align*}
The monotonicity property of the Fadell-Rabinowitz index (see Appendix~\ref{a:Fadell_Rabinowitz}) implies that 
\begin{align*}
 \FR(W_*)\geq\FR(\WW_c).
\end{align*}
The vector space $F_-$ has complex dimension $nN$, and the circle $S^1$ acts on it by rotations without fixed points. 
We set $Y=Y_N:=F_-\oplus W$.
Lemma~\ref{l:stabilization} and $\WW_c \neq \varnothing$ implies 
\begin{align*}
 \FR(Y_*)=\FR(W_*)+nN\geq\FR(\WW_c)+nN > nN.
\end{align*}

We claim that, if we fix $N\geq N_c$ large enough, we have
\begin{align}
\label{e:EH_hard_inclusion}
\big(\Pi\circ\Theta(Y\oplus F_0)\big) \cap S_+ \subset \UU.
\end{align}
Indeed, assume by contradiction that there exists a sequence $\gamma_N\in Y_N\oplus F_0$ such that $\Pi_N\circ\Theta(\gamma_N)\in S_+\setminus\UU$ for arbitrarily large $N\geq N_c$. We write \[\gamma_N=\gamma_{N,-}+\gamma_{N,0}+\gamma_{N,+},\] where $\gamma_{N,\pm}=\Pi_{N,\pm}(\gamma_N)$.
Since $(\Pi_{N,-}+\Pi_{N,0})\circ\Theta(\gamma_N)=0$ and $\Theta$ is of the form \eqref{e:Theta}, we have
\begin{align}\label{e:bounded_sequence_EH}
 \gamma_{N,-}+\gamma_{N,0}+ \big(e^{-\theta_-(\gamma_N)}\Pi_{N,-} + \Pi_0 \big)\KK(\gamma_N)=0.
\end{align}
Since $\theta_-$ and $\KK$ maps bounded sets to precompact sets and vanish outside a bounded subset of $E$, Equation~\eqref{e:bounded_sequence_EH} implies that the sequence $\gamma_{N,-}+\gamma_{N,0}$ is precompact. Since $\Theta$ is the identity outside a bounded set of $E$, the sequence $\gamma_{N,+}$ is also bounded; otherwise, after extracting a subsequence we would have $\|\gamma_{N,+}\|_E\to\infty$ as $N\to\infty$, and for $N$ large enough we would have $\gamma_{N}=\Theta(\gamma_N)$ and obtain the contradiction
\begin{align*}
 1
 =
 \|\Pi_N\circ\Theta(\gamma_N)\|_E
 =
 \|\gamma_N\|_E
 \geq
 \|\gamma_{N,+}\|_E\toup_{N\to\infty} \infty.
\end{align*}
Notice that $\gamma_{N,+}$ is contained in $\Pi_{N,+}(Y_N)=\Pi_{N,+}(\WW_c)=\R_+\Pi_{N,+}(\VV_c)$, and we recall that $\Pi_{N,+}(\VV_c)$ is compact (Lemma~\ref{l:compactness_EH}). Since the sequence $\gamma_{N,+}$ is uniformly bounded, it is precompact in $\Pi_{N,+}(\WW_c)\cup\{0\}$. Overall, we showed that the sequence $\gamma_N$ is precompact, and after extracting a subsequence we have a convergence
\[\lim_{N\to\infty}\gamma_N=\gamma\in \big(E_-\oplus E_0\oplus\WW_c\cup\{0\}\big)\setminus\UU.\] 
But $\Theta(\gamma)\in S_+$, and since $\Theta(0)=0$ we must have $\gamma\neq0$, and therefore $\gamma\in E_-\oplus E_0\oplus\WW_c$. This gives a contradiction, since $\UU$ is an open neighborhood of $\Theta(E_-\oplus E_0\oplus\WW_c)\cap S_+$, and completes the proof of \eqref{e:EH_hard_inclusion}.

We introduce the $S^1$-equivariant continuous map 
$\phi:=\Pi\circ\Theta|_{F} : F\to F$, which satisfies 
\begin{align*}
\phi(Y\oplus F_0) \cap S_+ \subset \UU
\end{align*}
according to \eqref{e:EH_hard_inclusion}. 
We denote by $B'\subset F$ the closed unit ball, which is $S^1$-invariant. Since the homeomorphism $\Theta$ is the identity on $E_-\oplus E_0$ and outside a ball of $E$ of sufficiently large radius, the map $\phi$ is the identity on $F_-\oplus F_0$ and outside a compact set. Therefore $B:=\phi^{-1}(B')\subset F$ is an $S^1$-invariant compact neighborhood of the origin.
We define another $S^1$-equivariant continuous map 
\[\psi:=(\id-\Pi_{+})\circ\phi|_{Y\oplus F_0}:Y\oplus F_0\to  F_-\oplus F_0,\]
which satisfies $\psi(0,x)=(0,x)$ for all $x\in F_0$. The $S^1$-invariant subset
\begin{align*}
Z:= \psi^{-1}(0)\cap\partial B,
\end{align*}
satisfies 
\[\phi(Z)\subset \phi(Y\oplus F_0)\cap S_+\subset\UU.\] Let $U\subset Y\oplus F_0$ be an $S^1$-invariant open subset containing $Z$ that is sufficiently small so that $\phi(U)\subset\UU$. Proposition~\ref{p:Borsuk_Ulam} applied to the map $\psi$ implies that $\FR(U)\geq\FR(Y_*)-nN$. Summing up, we obtained the desired lower bound
\[
 \FR(\UU)
 \geq
 \FR(U)
 \geq
 \FR(Y_*)-nN
 \geq
 \FR(\WW_c).
 \qedhere
\]
\end{proof}

\begin{proof}[Proof of Proposition~\ref{p:capacity<spectral}]
Assume that $s_k(B)<c$, so that 
\[\FR(\WW_c)\geq\FR(\{\Psi\leq c\})\geq k.\] 
By Lemma~\ref{l:Sikorav2}, for each $\epsilon>0$ there exists an admissible Hamiltonian $H\in\FF(B)$ such that
\begin{align*}
 \sup_{E_-\oplus E_0\oplus \WW_c} \Phi_H\leq c+\epsilon.
\end{align*}
Lemma~\ref{l:E-_E0_X} implies 
\[\EH(E_-\oplus E_0\oplus \WW_c)\geq \FR(\WW_c)\geq k.\]
Therefore $c_k(B)\leq c_k(H)\leq c+\epsilon$. Since $\epsilon>0$ can be chosen arbitrarily small, we infer that $c_k(B)\leq c$. Finally, since this latter inequality holds for all $c>s_k(B)$, we conclude $c_k(B)\leq s_k(B)$. 
\end{proof}

\subsection{Local maximizers of the capacity ratio}\label{ss:local_max_chat_k}

Let us extend the notion of Besse convex body to the class of smooth star-shaped domains. A $2n$-dimensional smooth star-shaped domain $B\subset\C^n$ is called \emph{Besse} when its boundary $\partial B$, equipped with the canonical contact form $\lambda|_{\partial B}$ of Equation~\eqref{e:Liouville}, is a Besse contact manifold: all its Reeb orbits are closed, and therefore have a common period according to Wadsley theorem \cite{Wadsley:1975aa}. Besse convex bodies, and in particular rational ellipsoids, are a subclass of Besse star-shaped domains. Actually, in dimension 4, these classes are the same up to symplectomorphisms.

\begin{Prop}\label{p:Besse_ellipsoids}
Let $B$ be a $4$-dimensional smooth star-shaped domain. The following conditions are equivalent:
\begin{itemize}
\setlength\itemsep{3pt}

\item[$(i)$] $B$ is Besse.  

\item[$(ii)$] $\partial B$ is strictly contactomorphic to the boundary of a rational ellipsoid $E$, i.e.~there exists a diffeomorphism $\psi:\partial B\to\partial E$ such that $\psi^*\lambda|_{\partial E}=\lambda|_{\partial B}$.  

\item[$(iii)$] There exists a diffeomorphism $\psi:\C^2\to \C^2$ such that $\psi(B)=E$, $\psi^*\omega=\omega$, and $\psi|_{\partial B}^*\lambda=\lambda|_{\partial B}$.  

\item[$(iv)$] $B$ is symplectomorphic to a rational ellipsoid $E$, i.e.~there exists a diffeomorphism $\psi:B\to E$ such that $\psi^*\omega=\omega$.  

\end{itemize}
\end{Prop}

\begin{proof}
Assume that $B$ is a $4$-dimensional Besse star-shaped domain. As it was pointed out in \cite[Theorem~1.1]{Mazzucchelli:2023aa}, the contact classification of Seifert fibrations \cite{Geiges:2018aa, Cristofaro-Gardiner:2020aa} implies that the boundary of $B$ is strictly contactomorphic to the boundary of a rational ellipsoid $E$: there exists a diffeomorphism $\psi:\partial B\to\partial E$ such that $\psi^*\lambda|_{\partial E}=\lambda|_{\partial B}$. A straightforward generalization of \cite[Prop.~4.3]{Abbondandolo:2018aa} implies that $\psi$ can be extended to a symplectomorphism $\psi\in\Symp(\C^2,\omega)$. Therefore conditions (i) and (ii) are equivalent, and  imply condition (iii).

Condition (iii) immediately implies condition (iv). Assume now that a 4-di\-men\-sion\-al smooth star-shaped domain $B$ is symplectomorphic to a rational ellipsoid $E$ via a symplectomorphism $\psi:B\to E$. The Reeb orbits on $\partial B$ and $\partial E$ are leaves of the characteristic foliations $\ker(\omega|_{\partial B})$ and $\ker(\omega|_{\partial E})$ respectively. Since  $\psi^*\omega=\omega$, the restriction $\psi|_{\partial B}$ maps Reeb orbits of $\partial B$ to time-reparametrized Reeb orbits of $\partial E$. Since all Reeb orbits of $\partial E$ are closed, all Reeb orbits of $\partial B$ are closed as well. Therefore condition (iv) implies condition (i).
\end{proof}

Let $B$ be a $4$-dimensional Besse star-shaped domain, and $\psi\in\Symp(\C^2,\omega)$ a symplectomorphism mapping $B$ to a rational ellipsoid $E:=\psi(B)$, and whose restriction $\psi|_{\partial B}:\partial B\to\partial E$ is a strict contactomorphism. Since $\psi|_{\partial B}$ conjugates the Reeb flows on $\partial B$ and $\partial E$, in particular $\sigma(\partial B)=\sigma(\partial E)$.
The monotonicity property of the Ekeland-Hofer capacities implies that $c_k(E)=c_k(B)$. In particular, the capacities $c_k(B)$ cover the whole spectrum $\sigma(\partial B)$, i.e.
\begin{align*}
 \sigma(\partial B)=\big\{c_k(B)\ \big|\ k\geq1\big\},
\end{align*}
since the same holds for the ellipsoid $E$. 
We denote by $\tau(B)$ the minimal common period of the closed Reeb orbit on $\partial B$. For all integers $m\geq1$, we define
\begin{align*}
k_m(B) & := \min\big\{ k\geq1\ \big|\  c_k(B)=m\,\tau(B)\big\},\\
K(B) & :=\big\{k_m(B)\ \big|\ m\geq1  \big\}.
\end{align*}
Clearly, $\tau(E)=\tau(B)$ and $K(E)=K(B)$.
These definitions of $k_m(B)$ and $K(B)$ agree with the ones in Section~\ref{ss:result} in the special case of rational ellipsoids. They also agree with the one given in Section~\ref{ss:local_max_sk} for Besse convex domains, thanks to the following lemma.

\begin{Lemma}
\label{l:ck=sk_Besse}
For each $4$-dimensional Besse convex body $B$, we have $c_k(B)=s_k(B)$ for all $k\geq 1$.
\end{Lemma}

\begin{proof}
Let $\psi\in\Symp(\C^2,\omega)$ be a symplectomorphism mapping $B$ to a rational ellipsoid $E:=\psi(B)$, so that $\tau=\tau(B)=\tau(E)$.  We already know that both the spectral invariants $s_k(B)$ and the Ekeland-Hofer capacities $c_k(B)=c_k(E)$ cover the whole action spectrum $\sigma=\sigma(\partial B)=\sigma(\partial E)$, i.e.
\begin{align*}
 \sigma = \{s_k(B)\ |\ k\geq1\} = \{c_k(E)\ |\ k\geq1\}.
\end{align*}
Moreover $s_1(B)=c_1(E)=\min\sigma$, $s_k(B)\leq s_{k+1}(B)$, and $c_k(E)\leq c_{k+1}(E)$.
The spectral Besse characterization for the Clarke action functional  (Section~\ref{ss:Besse}) implies that $s_1(B)=\tau$ if and only if $s_1(B)=s_{2}(B)$; moreover, for any $k\geq2$, it implies that $s_k(B)$ is a multiple of $\tau$ if and only if either $s_{k-1}(B)=s_k(B)$ or $s_{k}(B)=s_{k+1}(B)$. These properties fully determine $s_k(B)$ for the Besse convex body $B$. The normalization property of the symplectic $k$-th capacities implies that  $c_k(E)$ satisfy analogous properties: $c_1(E)=\tau$ if and only if $c_1(E)=c_{2}(E)$; for any $k\geq2$,  $c_k(E)$ is a multiple of $\tau$ if and only if either $c_{k-1}(E)=c_k(E)$ or $c_{k}(E)=c_{k+1}(E)$. We conclude that $s_k(B)=c_k(E)$.
\end{proof}

We infer the first half of Theorem~\ref{mt:local_max} from Theorem~\ref{t:main_Clarke}.

\begin{Lemma}
\label{l:star_Besse_loc_max}
Let $B_0$ be a $4$-dimensional Besse star-shaped domain. For each integer $k\in K(B_0)$, any smooth star-shaped domain $B_1$ that is sufficiently $C^3$-close to $B_0$ satisfies $\chat_k(B_1)\leq \chat_k(B_0)$.
\end{Lemma}

\begin{proof}
Fix $k\in K(B_0)$, and a symplectomorphism $\psi\in\Symp(\C^2,\omega)$ mapping $B_0$ to a rational ellipsoid. If a smooth star-shaped domain $B_1$ is $C^3$-close to $B_0$, then $\psi(B_1)$ is $C^3$-close to $\psi(B_0)$. In particular, $\psi(B_1)$ is a smooth convex body, and Theorem~\ref{t:main_Clarke} implies $\widehat s_k(\psi(B_1))\leq \widehat s_k(\psi(B_0))$. Proposition~\ref{p:capacity<spectral} implies $c_k(\psi(B_1))
 \leq 
 s_k(\psi(B_1))$, and Lemma~\ref{l:ck=sk_Besse} implies $c_k(\psi(B_0))
 = 
 s_k(\psi(B_0))$. Therefore we conclude
\[
 \chat_k(B_1)
 =
 \chat_k(\psi(B_1))
 \leq 
 \widehat s_k(\psi(B_1))
 \leq
 \widehat s_k(\psi(B_0))
 =
 \chat_k(\psi(B_0))
 =
 \chat_k(B_0).
\qedhere
\]
\end{proof}

Next, we address the opposite implication in Theorem~\ref{mt:local_max}. The following lemma holds in arbitrary dimension.

\begin{Lemma}
\label{l:max_are_Besse}
Let $B_0$ be a local maximizer of the capacity ratio $\chat_k$ over the space of $2n$-dimensional smooth star-shaped domains endowed with the $C^\infty$ topology. Then $B_0$ is a Besse star-shaped domain and $c_k(B_0)$ is a common period for the Reeb orbits on $\partial B_0$.
\end{Lemma}

\begin{proof}
The proof builds on an argument originally due to Alvarez Paiva and Balacheff \cite{Alvarez-Paiva:2014aa}.
Let $B_0$ be a $2n$-dimensional smooth star-shaped domain, and $h_0:\C^n\to[0,\infty)$  the 2-homogeneous Hamiltonian such that $h_0^{-1}(1)=\partial B_0$. The restriction of its Hamiltonian flow $\phi_{h_0}^t|_{\partial B_0}$ is the Reeb flow of $\partial B_0$. Assume by contradiction that $c_k(B_0)$ is not a common period for the Reeb orbits on the boundary $\partial B_0$. Namely, there exists $\delta>0$ and a non-empty open subset $U\subset\partial B_0$ such that 
\[\phi_h^{t}(z)\neq z,\qquad\forall z\in U,\ t\in[c_k(B_0)-\delta,c_k(B_0)+\delta].\] 
Let $\chi:\partial B_0\to[0,\infty)$ be a smooth function supported in $U$ and not identically zero. For $\epsilon>0$, let $h_\epsilon:\C^n\to[0,\infty)$ be the 2-homogeneous Hamiltonian such that $h_\epsilon|_{\partial B_0}=h_0|_{\partial B_0}+\epsilon\chi$. For all $\epsilon>0$ small enough, the sublevel set $B_\epsilon:=h_\epsilon^{-1}[0,1]$ is a smooth star-shaped domain strictly contained in $B_0$, and in particular 
\[\vol(B_\epsilon)<\vol(B_0).\] 
By our choice of $\chi$, for all $\epsilon>0$ small enough we have
\begin{align*}
 \fix(\phi_{h_\epsilon}^t)= \fix(\phi_{h_0}^t),\qquad\forall t\in[c_k(B_0)-\delta,c_k(B_0)+\delta].
\end{align*}
Therefore the continuous function $\epsilon\mapsto c_k(B_\epsilon)$ takes values into the intersection $\sigma(\partial B_0)\cap[c_k(B_0)-\delta,c_k(B_0)+\delta]$ for all $\epsilon\geq 0$ small enough. Since the action spectrum $\sigma(\partial B_0)$ is  nowhere dense (see, e.g., \cite[Prop.~7.4]{Sikorav:1990aa}), we conclude that $c_k(B_\epsilon)=c_k(B_0)$, and therefore $\chat_k(B_\epsilon)>\chat_k(B_0)$, for all $\epsilon>0$ small enough.
\end{proof}

\begin{proof}[Proof of Theorem~\ref{mt:local_max}]
We consider the capacity ratio $\chat_k$ over the space of 4-dimen\-sional smooth star-shaped domains, endowed with the $C^3$ topology. Let $B$ be a 4-dimensional smooth star-shaped domain symplectomorphic to a rational ellipsoid $E$. Proposition~\ref{p:Besse_ellipsoids} implies that $B$ is Besse, and Lemma~\ref{l:star_Besse_loc_max} asserts that  $B$ is a local maximizer of $\chat_k$ for all $k\in K(B)=K(E)$. 

Conversely, assume that a 4-dimensional smooth star-shaped domains $B_0$ is a local maximizer of $\chat_k$. Lemma~\ref{l:max_are_Besse} implies that $B_0$ is Besse, and therefore Proposition~\ref{p:Besse_ellipsoids} implies that there exists $\psi\in\Symp(\C^2,\omega)$ mapping $B_0$ to a rational ellipsoid $E(\aaa)$. For each $\bbb$ sufficiently close to $\aaa$, the preimage $B_1:=\psi^{-1}(E(\bbb))$ is a smooth star-shaped domain $C^3$-close to $B_0$, and therefore
\[
\chat_k(E(\bbb)) = \chat_k(B_1) \leq \chat_k(B_0) = \chat(E(\aaa)).
\]
Namely, $E(\aaa)$ is a local maximizer of $\chat_k$ over the space of 4-dimensional ellipsoids, and Proposition~\ref{mp:ellipsoids}(i) implies that $k\in K(E(\aaa))=K(B_0)$.
\end{proof}

We can also complete Theorem~\ref{t:main_Clarke}, fully characterizing the local maximizers of the spectral ratios $\widehat s_k$ over the space of 4-dimentional smooth convex domains.

\begin{Thm}\label{t:final_Clarke}
 On the space of $4$-dimensional smooth convex domains endowed with the $C^3$ topology, the local maximizers of $B\mapsto \widehat s_k(B)$ are precisely those domains symplectomorphic to a $2n$-dimensional rational ellipsoid $E$ with $k\in K(E)$.
\end{Thm}

\begin{proof}
Theorem~\ref{t:main_Clarke}, together with Proposition~\ref{p:Besse_ellipsoids}, provides one implication: the smooth convex bodies $B$ symplectomorphic to a $4$-dimensional rational ellipsoid $E$ are local maximizers of $\widehat s_k$ for all $k\in K(B)=K(E)$.

Conversely, let $B_0$ be a smooth convex domain that is a local maximizer of $\widehat s_k$. An analogous argument to the one in the proof of Lemma~\ref{l:max_are_Besse} implies that $B_0$ is Besse, and Proposition~\ref{p:Besse_ellipsoids} provides a symplectomorphism $\psi\in\Symp(\C^2,\omega)$ mapping $B_0$ to a rational ellipsoid $E(\aaa)$. For each $\bbb$ sufficiently close to $\aaa$, the preimage $B_1:=\psi^{-1}(E(\bbb))$ is a smooth convex body $C^3$-close to $B_0$, and therefore $\widehat s_k(B_1)\leq \widehat s_k(B_0)$.  Proposition~\ref{p:capacity<spectral} implies $c_k(B_1)\leq s_k(B_1)$, and Lemma~\ref{l:ck=sk_Besse} implies $s_k(B_0) = c_k(B_0)$. Overall, we obtained
\begin{align*}
\chat_k(E(\bbb))
=
\chat_k(B_1)
\leq
\widehat s_k(B_1)
\leq
\widehat s_k(B_0)
=
\chat_k(B_0)
=
\chat_k(E(\aaa)).
\end{align*}
Namely, $E(\aaa)$ is a local maximizer of $\chat_k$ over the space of 4-dimensional ellipsoids.  Proposition~\ref{mp:ellipsoids}(i) implies that $k\in K(E(\aaa))=K(B_0)$.
\end{proof}

\section{Maximizers of the capacity ratios among ellipsoids}\label{s:ellipsoids}

We recall that, throughout this paper, we compute the volumes of domains in $\C^n$ by integrating the volume form $\omega^n$; namely, for a $2n$-dimensional domain $K$, the volume $\vol(K)=\vol(K,\omega)$ is $n!$ times the usual $2n$-dimensional Euclidean volume. For a $2n$-dimensional ellipsoid $E(\aaa)$, with $\aaa=(a_1,...,a_n)$, we have
\begin{align*}
\vol(E(\aaa))=a_1...a_n
\end{align*}
All symplectic $k$-capacities coincide over the space of ellipsoids, according to their normalization property (see Section~\ref{ss:result}). Throughout this section, we denote by $c_k$ any symplectic $k$-capacity.

\begin{proof}[Proof of Proposition~\ref{mp:ellipsoids}\,$(i)$]
Let $\aaa=(a_1,...,a_n)$ be a local maximizer of the function $\aaa\mapsto\chat_k(E(\aaa))$, and assume without loss of generality that $0<a_1\leq...\leq a_n<\infty$. We first prove that every $a_i$ divides $c_k(E(\aaa))$. Indeed, if we had $r a_i< c_k(E(\aaa))< (r+1)a_i$ for some integer $r\geq0$ and some minimal $i\in\{1,...,n\}$, then for any $t<1$ sufficiently close to $1$ the configuration $\bbb(t)=(b_1(t),...,b_n(t))$ given by
\begin{align*}
b_j(t)
= 
\left\{
  \begin{array}{@{}ll}
    t\,a_i, & \mbox{ if }  j=i, \vspace{5pt} \\
    a_j, & \mbox{ if } j\neq i \\ 
  \end{array}
\right.
\end{align*}
would satisfy $c_k(E(\bbb(t)))=c_k(E(\aaa))$ and $\vol(E(\bbb(t)))<\vol(E(\aaa))$, contradicting the maximality of $\aaa$. 

Therefore, $c_k(E(\aaa)) = r_1 a_1 = ... = r_n a_n$
for some integers $r_1\geq r_2\geq ... \geq r_n\geq 1$. We now prove that \[c_k(\aaa)=c_{k+n-1}(\aaa),\]
which readily implies $r_1+ ... + r_n = k+n-1$ and therefore $k\in K(E(\aaa))$.
Assume by contradiction that $c_k(\aaa)<c_{k+n-1}(\aaa)$. This implies that $c_k(E(\aaa))=c_{k-1}(E(\aaa))$. For any $t<1$ sufficiently close to $1$, the configuration $\bm b(t)=(ta_1,a_2,...,a_n))$ has capacity $c_k(E(\bbb(t)))=c_k(E(\aaa))$, but lower volume $\vol(E(\bbb(t)))<\vol(E(\aaa))$, contradicting the maximality of $\aaa$.

We are left to prove the converse implication. Let $E(\aaa)$ be a rational ellipsoid. As usual, assume without loss of generality that $a_1\leq...\leq a_n$. Fix any integer $k\in K(E(\aaa))$. There exist integers $r_i\geq1$ such that $r_1a_1=...=r_na_n$ and $k=r_1+...+r_n-n+1$. For all $\bbb$ sufficiently close to $\aaa$, we have $\chat_k(E(\bbb))=f(\bbb)$, where $f$ is the $0$-homogeneous function
\begin{align*}
 f(\bbb)=\frac{\min\{ r_1b_1,...,r_nb_n \}}{(b_1...b_n)^{1/n}}.
\end{align*}
As usual, assume without loss of generality that $b_1\leq...\leq b_n$.
We set
\[t := \frac{r_1a_1}{\min\{ r_1b_1,...,r_nb_n \}}.\]
Notice that $t\,b_i\geq a_i$. Therefore, for all $\bbb$ sufficiently close to $\aaa$, we have
\[
\chat_k(E(\bbb))
=
f(\bbb)
=
f(t\,\bbb)
=
\frac{r_1a_1}{(t\,b_1...t\,b_n)^{1/n}}
\leq
\frac{r_1a_1}{(a_1...a_n)^{1/n}}
=
f(\aaa)=\chat_k(E(\aaa)).
\qedhere
\]
\end{proof}

\begin{proof}[Proof of Proposition~\ref{mp:ellipsoids}\,$(ii)$]
For configurations $\aaa=(a_1,...,a_n)$ with 
\begin{align}
\label{e:ordered_configuration}
0<a_1\leq...\leq a_n<\infty, 
\end{align}
we have
\[\chat_k(E(\aaa))\leq\frac{ka_1}{(a_1...a_n)^{1/n}}\ttoup_{a_n\to\infty}0. \]
Therefore, the function $\aaa\mapsto\chat_k(E(\aaa))$ achieves its maximum at some $\aaa$, which we assume to satisfy~\eqref{e:ordered_configuration} without loss of generality. By Proposition~\ref{mp:ellipsoids}(i), we have $k\in K(E(\aaa))$. Namely
\begin{align*}
c_k(E(\aaa)) = c_{k+n-1}(E(\aaa)) = s_1 a_1 = ... = s_n a_n
\end{align*}
for some integers $s_1\geq s_2\geq ... \geq s_n\geq 1$. 

We claim that 
\begin{align*}
s_i-s_j\leq 1,
\qquad
\forall 1\leq i<j\leq n.
\end{align*}
Otherwise, we can find $i<j$ such that $s_i-s_j\geq 2$, $s_i-s_{i+1}\geq1$, and $s_{j-1}-s_j\geq1$ (the case $i=j-1$ is allowed). The configuration $\bm b=(b_1,...,b_n)$ given by
\begin{align*}
b_l
= 
\left\{
  \begin{array}{@{}ll}
    \frac{s_i}{s_i-1}a_i, & \mbox{ if }  l=i, \vspace{5pt} \\
    \frac{s_j}{s_j+1}a_j, & \mbox{ if }  l=j, \vspace{5pt} \\
    a_l, & \mbox{ if } l\neq i,j \\ 
  \end{array}
\right.
\end{align*}
has the same value $c_k(E(\bbb))=c_k(E(\aaa))$, but 
\begin{align*}
\vol(E(\bbb))
= 
\frac{s_i s_j}{(s_i-1)(s_j+1)} \vol(E(\aaa))
=
\frac{s_i s_j}{s_is_j + \underbrace{s_i -s_j -1}_{\geq1}} \vol(E(\aaa))
<
\vol(E(\aaa)),
\end{align*}
contradicting the maximality of $\aaa$.

Summing up, we proved that
\begin{align*}
c_k(E(\aaa)) 
& =
c_{k+n-1}(E(\aaa)) 
= (s+2) a_1 
= (s+2) a_2
= ... 
= (s+2) a_{i-1}\\
& = 
(s+1) a_i 
= (s+1) a_{i+1}
=...
= (s+1) a_n
\end{align*}
for some integers $s\geq 0$ and $i\in\{1,...,n\}$. Since $c_k(E(\aaa))=c_{k+n-1}(E(\aaa))$, we have 
\begin{align*}
k+n-1= (s+2)(i-1) + (s+1)(n-i+1),
\end{align*}
that is,
$k = sn + i$,
which implies that $q=s$ and $r=i$.
\end{proof}

\section{Maximizers of the capacity ratios among toric domains}\label{s:toric}

Throughout this section, we denote by $c_k$ any symplectic $k$-capacity satisfying the extra ``closed Reeb orbits'' assumption (see Section~\ref{ss:result}). We shall consider 4-dimensional toric domains $X_\Omega$ whose profile $\Omega$ is the subgraph of monotone non-increasing functions $f:[0,x_0]\to[0,\infty)$ such that $f|_{[0,x_0)}>0$ and $f(x_0)=0$, i.e.
\begin{align*}
\Omega=\Omega_f:=\big\{ (x,y) \in \R^2\ \big|\ x\in[0,x_0],\ y\in[0,f(x)] \big\}.
\end{align*}

\subsection{Concave toric domains}
Let $X_\Omega$ be a 4-dimensional concave toric domain. Its profile $\Omega=\Omega_f\subset[0,\infty)^2$ must be the subgraph of a (monotone decreasing) convex function $f:[0,x_0]\to[0,\infty)$ such that $f|_{[0,x_0)}>0$ and $f(x_0)=0$.
As usual, we denote by $\vol(X_\Omega)=\vol(X_\Omega,\omega)$ the volume of $X_\Omega$ computed by integrating the volume form $\omega^2$. This volume $\vol(X_\Omega)$ is twice the Euclidean area of $\Omega$.

Gutt and Hutchings \cite[Theorem~1.14]{Gutt:2018aa} provided a formula for the higher capacities of $2n$-dimensional concave toric domains. For our 4-dimensional $X_\Omega$ the formula reads:
\begin{align}
\label{e:GH_concave}
c_k(X_\Omega)
=
\max\Big\{ 
 [v]_\Omega\ 
\Big|\ v=(v_1,v_2)\in\{1,...,k\}^2,\ v_1+v_2=k+1
\Big\},
\end{align}
where 
\begin{align*}
[v]_\Omega=\min\big\{\langle v,w\rangle\ \big|\ w\in\mathrm{graph}(f)\big\}. 
\end{align*}
Proposition~\ref{mp:toric}(i) is a direct consequence of this formula and of the properties of symplectic $k$-capacities.

\begin{proof}[Proof of Proposition~\ref{mp:toric}$(i)$]
Let $v\in\{1,...,k\}^2$ be a vector such that $c_k(X_\Omega)=[v]_\Omega$, and $w=(w_1,w_2)\in\mathrm{graph}(f)$ be a vector such that $[v]_\Omega=\langle v,w\rangle$. We denote by $g:\R\to\R$ an affine function such that $g(w_1)=w_2$ and $g\leq f$. We set 
$a_2:=g(0)$ and $a_1:=a_2/g'(0)$, so that $g(a_1)=0$.
The ellipsoid $E(a_1,a_2)=X_{\Upsilon}$ is a concave toric domain whose associated profile $\Upsilon=\Omega_g$ satisfies $\Upsilon\subseteq\Omega$, and therefore
\begin{align*}
\vol(E(a_1,a_2))=2\, \area(\Upsilon)\leq2\,\area(\Omega)=\vol(X_\Omega).
\end{align*}
Moreover, by the formula~\eqref{e:GH_concave} applied to $E(a_1,a_2)$, we have
\begin{align*}
c_k(E(a_1,a_2))
\geq
[v]_\Upsilon = [v]_\Omega = c_k(X_\Omega).
\end{align*}
Overall, we obtained
$\chat_k(E(a_1,a_2))\geq\chat_k(X_\Omega)$. Finally, by Proposition~\ref{mp:ellipsoids}, the function $(a_1,a_2)\mapsto\chat_k(E(a_1,a_2))$ achieves its maximum at \[(a_1,a_2)=(\big\lceil\tfrac k2\big\rceil,\big\lceil\tfrac {k+1}2\big\rceil\big).
\qedhere\]
\end{proof}

\subsection{Convex toric domains}
Let $X_\Omega$ be a 4-dimensional convex toric domain. Its profile $\Omega=\Omega_f\subset[0,\infty)^2$ must be the subgraph of a concave and non-increasing function $f:[0,x_0]\to[0,y_0]$ such that $f(0)=y_0>0$ and $f(x_0)=0$. We consider the convex body 
\[\widehat\Omega:=\big\{(x_1,x_2)\in\R^2\ \big|\ (|x_1|,|x_2|)\in\Omega\big\}.\]
We denote by $\|\cdot\|_\Omega$ the norm on $\R^2$ whose unit ball is the polar of $\widehat\Omega$, i.e.
\begin{align*}
 \|v\|_\Omega := \max_{w\in\Omega}\ \langle v,w\rangle,\qquad\forall v\in[0,\infty)^2.
\end{align*}
We set $V_k:=\big\{v_0,...,v_k\big\}$, where $v_j=(j,k-j)\in[0,k]^2$. 
Gutt and Hutchings \cite[Theorem~1.16]{Gutt:2018aa} provided a formula for the higher capacities of $2n$-dimensional convex toric domains, which in the special case of our 4-dimensional $X_\Omega$ reads:
\begin{align*}
c_k(\Omega)
=
\min_{v\in V_k}
\ \|v\|_\Omega.
\end{align*}

\begin{Example}\label{ex:polydisk}
For each $a,b>0$, the polydisk $P(a,b):=E(a)\times E(b)=X_\Omega$ is a convex toric domain with profile $\Omega=[0,a]\times[0,b]$. Therefore $c_k(P(a,b))=k\min\{a,b\}$, $\vol(P(a,b))=2ab$, and
\begin{align*}
 \chat_k(P(a,b))
 =
 k\frac{\min\{a,b\}}{\sqrt{2ab}}.
\end{align*}
The polydisk that maximizes $\chat_k$ is (up to rescaling) $P(1,1)$, for which $\chat_k(P(1,1))=k/\sqrt{2}$.
\hfill\qed
\end{Example}

The proof of Proposition~\ref{mp:toric}(ii) requires several lemmas.

\begin{Lemma}
Over the space of $4$-dimensional convex toric domains, the capacity ratios $\chat_k$ admit a global maximizer.
\end{Lemma}

\begin{proof}
Consider a profile $\Omega=\Omega_f$ that is the subgraph of a concave and non-increasing function $f:[0,x_0]\to[0,y_0]$ such that $f(0)=y_0$ and $f(x_0)=0$. Since the capacity ratios are scale invariant, we can always assume that $y_0=1$. 
Assume in addition that $x_0\leq1$. The volume of the associated convex toric domain $X_\Omega$ is bounded from below as 
\begin{align*}
 \vol(X_\Omega)=2\,\area(\Omega)\geq x_0.
\end{align*}
The $k$-th capacity of $X_\Omega$ is bounded from above as
\begin{align*}
c_k(X_\Omega)
=
\min_{j\in\{0,...,k\}}\max_{w\in\Omega}\ \langle v_j,w\rangle
\leq
\min_{j\in\{0,...,k\}} \big( j x_0 + (k-j) \big)
\leq k x_0.
\end{align*}
Therefore $\chat_k(X_\Omega)\leq k\sqrt{x_0}<\chat_k(P(1,1))$ if $x_0<1/2$. Analogously, $\chat_k(X_\Omega)\leq k/\sqrt 2$ if $x_0>2$.

Let $\mathcal{T}$ be the space of 4-dimensional convex toric domains, and $\mathcal{T'}\subset\mathcal{T}$ the subspace of those $X_{\Omega_f}$ such that $f:[0,x_0]\to[0,1]$ is concave, non-increasing, and satisfies $f(0)=1$ and $x_0\in[1/2,2]$. The conclusion of the previous paragraph implies
\begin{align*}
\sup_{\mathcal{T}}\ \chat_k
=
\sup_{\mathcal{T}'}\ \chat_k.
\end{align*}
The space $\mathcal{T}'$ is compact for the Hausdorff topology, and since $\chat_k|_{\mathcal{T}}$ is continuous, its restriction to $\mathcal{T}'$ achieves its maximum.
\end{proof}

We set
\begin{align*}
\Delta_\pm := \big\{ (x,y)\in(0,\infty)^2\ \big|\ \pm x<\pm y \big\},\qquad
\Delta := \big\{ (x,x)\in\R^2\ \big|\ x>0 \big\},
\end{align*}
and
\begin{align*}
\partial_\pm\Omega :=\Delta_\pm\cap\partial\Omega,\qquad
 \partial_0\Omega :=\Delta\cap\partial\Omega,\qquad
 \partial_*\Omega :=\partial_+\Omega\cup\partial_0\Omega\cup\partial_-\Omega.
\end{align*}
Notice that $\partial_0\Omega$ is a singleton, and $\Omega$ is the convex set enclosed by $\partial_*\Omega$ together with two segments on the $x$ and $y$ axes.

\begin{Lemma}\label{l:bound_corners}
Assume that $X_\Omega$ maximizes $\chat_k$ over the space of convex toric domains. Then $\partial_*\Omega$ is piecewise linear with at most $k-1$ corners.
\end{Lemma}

\begin{proof}
For each $v\in V_k$ there exists $w_v\in\partial_*\Omega$ such that $\|v\|_\Omega=\langle v,w_v\rangle$. The point $w_v$ is not necessarily unique, but for $v=(0,k)$ we choose $w_v=(0,y_0)$, and for $v=(k,0)$ we choose $w_v=(x_0,0)$. We denote by $\Omega'\subset\Omega$ the convex hull of the finite set 
\[\big\{(0,0)\big\}\cup\big\{w_v\ \big|\ v\in V_k\big\}\subset\partial\Omega.\]
Notice that $X_{\Omega'}$ is a convex toric domain with $c_k(\Omega')= c_k(\Omega)$. Therefore, $\chat_k(\Omega')\geq\chat_k(\Omega)$, with equality if and only if $\Omega=\Omega'$. This shows that $\partial_*\Omega$ is piecewise linear with at most $k-1$ corners.
\end{proof}

Lemma~\ref{l:bound_corners}, together with Proposition~\ref{mp:ellipsoids}, implies the claim of Proposition~\ref{mp:toric}(ii) in the case $k=1$. Hence, from now on we consider $k\geq 2$. Moreover, in view of Lemma~\ref{l:bound_corners}, from now on we only need to consider convex domains $X_\Omega$ whose profile $\Omega$ has piecewise linear boundary.

\begin{Lemma}
\label{l:v_w}
Assume that $X_\Omega$ maximizes $\chat_k$ over the space of convex toric domains. For each corner $w$ of $\partial_*\Omega$, there exists $v\in V_k$ such that $c_k(X_\Omega)=\|v\|_\Omega=\langle v,w\rangle$. 
\end{Lemma}

\begin{proof}
If such a $v$ does not exist, by chamfering $\Omega$ near $w$ we would produce $\Omega'\subsetneq\Omega$ still defining a convex toric domain $X_{\Omega'}$ with the same capacity $c_k(X_{\Omega'})=c_k(X_{\Omega})$ but smaller volume $\vol(X_{\Omega'})<\vol(X_\Omega)$, in contradiction to the fact that $X_\Omega$ maximizes $\chat_k$.
\end{proof}

\begin{Lemma}
\label{l:corners}
Assume that $X_\Omega$ maximizes $\chat_k$ over the space of convex toric domains. Then $\partial_*\Omega$ is piecewise linear with at most two corners. Moreover:
\begin{itemize}

\item[$(i)$] If $\partial_*\Omega$ has no corners, then $X_\Omega$ is  an ellipsoid.

\item[$(ii)$] If $w=(x,x)\in\partial_0\Omega$ is a corner of $\partial_*\Omega$, then it is its only corner and $X_\Omega=P(x,x)$.

\item[$(iii)$] $\partial_+\Omega$ contains at most one corner $w=(x,y)$ of $\partial_*\Omega$, and if such a corner exists then $y=y_0$.

\item[$(iv)$] $\partial_-\Omega$ contains at most one corner $w=(x,y)$ of $\partial_*\Omega$, and if such a corner exists then $x=x_0$.
 
\end{itemize}
The possible situations allowed by the last two points are described by the following pictures:
\begin{center}
\scriptsize
\begingroup%
  \makeatletter%
  \providecommand\color[2][]{%
    \errmessage{(Inkscape) Color is used for the text in Inkscape, but the package 'color.sty' is not loaded}%
    \renewcommand\color[2][]{}%
  }%
  \providecommand\transparent[1]{%
    \errmessage{(Inkscape) Transparency is used (non-zero) for the text in Inkscape, but the package 'transparent.sty' is not loaded}%
    \renewcommand\transparent[1]{}%
  }%
  \providecommand\rotatebox[2]{#2}%
  \newcommand*\fsize{\dimexpr\f@size pt\relax}%
  \newcommand*\lineheight[1]{\fontsize{\fsize}{#1\fsize}\selectfont}%
  \ifx\svgwidth\undefined%
    \setlength{\unitlength}{380.69870068bp}%
    \ifx\svgscale\undefined%
      \relax%
    \else%
      \setlength{\unitlength}{\unitlength * \real{\svgscale}}%
    \fi%
  \else%
    \setlength{\unitlength}{\svgwidth}%
  \fi%
  \global\let\svgwidth\undefined%
  \global\let\svgscale\undefined%
  \makeatother%
  \begin{picture}(1,0.1938822)%
    \lineheight{1}%
    \setlength\tabcolsep{0pt}%
    \put(0,0){\includegraphics[width=\unitlength,page=1]{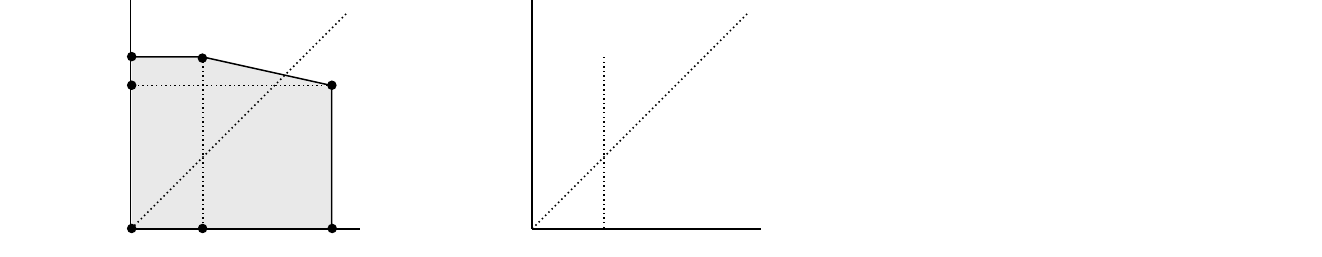}}%
    \put(0.45211597,0.1571409){\color[rgb]{0,0,0}\makebox(0,0)[lt]{\lineheight{1.25}\smash{\begin{tabular}[t]{l}$w$\end{tabular}}}}%
    \put(0,0){\includegraphics[width=\unitlength,page=2]{lemma.pdf}}%
    \put(0.24407755,0.00505208){\color[rgb]{0,0,0}\makebox(0,0)[lt]{\lineheight{1.25}\smash{\begin{tabular}[t]{l}$x_0$\end{tabular}}}}%
    \put(0.07359104,0.14697583){\color[rgb]{0,0,0}\makebox(0,0)[lt]{\lineheight{1.25}\smash{\begin{tabular}[t]{l}$y_0$\end{tabular}}}}%
    \put(0.08690408,0.00407453){\color[rgb]{0,0,0}\makebox(0,0)[lt]{\lineheight{1.25}\smash{\begin{tabular}[t]{l}$0$\end{tabular}}}}%
    \put(0.17965244,0.06888043){\color[rgb]{0,0,0}\makebox(0,0)[lt]{\lineheight{1.25}\smash{\begin{tabular}[t]{l}\normalsize$\Omega$\end{tabular}}}}%
    \put(0.14872654,0.00505208){\color[rgb]{0,0,0}\makebox(0,0)[lt]{\lineheight{1.25}\smash{\begin{tabular}[t]{l}$x_1$\end{tabular}}}}%
    \put(0.07359104,0.12530513){\color[rgb]{0,0,0}\makebox(0,0)[lt]{\lineheight{1.25}\smash{\begin{tabular}[t]{l}$y_2$\end{tabular}}}}%
    \put(0.14872654,0.1571409){\color[rgb]{0,0,0}\makebox(0,0)[lt]{\lineheight{1.25}\smash{\begin{tabular}[t]{l}$w_1$\end{tabular}}}}%
    \put(0.24915776,0.13621699){\color[rgb]{0,0,0}\makebox(0,0)[lt]{\lineheight{1.25}\smash{\begin{tabular}[t]{l}$w_2$\end{tabular}}}}%
    \put(0.54746695,0.00505208){\color[rgb]{0,0,0}\makebox(0,0)[lt]{\lineheight{1.25}\smash{\begin{tabular}[t]{l}$x_0$\end{tabular}}}}%
    \put(0.37698058,0.14697583){\color[rgb]{0,0,0}\makebox(0,0)[lt]{\lineheight{1.25}\smash{\begin{tabular}[t]{l}$y_0$\end{tabular}}}}%
    \put(0.39028739,0.00407453){\color[rgb]{0,0,0}\makebox(0,0)[lt]{\lineheight{1.25}\smash{\begin{tabular}[t]{l}$0$\end{tabular}}}}%
    \put(0.43970052,0.06888043){\color[rgb]{0,0,0}\makebox(0,0)[lt]{\lineheight{1.25}\smash{\begin{tabular}[t]{l}\normalsize$\Omega$\end{tabular}}}}%
    \put(0.45211597,0.00505208){\color[rgb]{0,0,0}\makebox(0,0)[lt]{\lineheight{1.25}\smash{\begin{tabular}[t]{l}$x$\end{tabular}}}}%
    \put(0.82918598,0.00505208){\color[rgb]{0,0,0}\makebox(0,0)[lt]{\lineheight{1.25}\smash{\begin{tabular}[t]{l}$x_0$\end{tabular}}}}%
    \put(0.6586995,0.14697583){\color[rgb]{0,0,0}\makebox(0,0)[lt]{\lineheight{1.25}\smash{\begin{tabular}[t]{l}$y_0$\end{tabular}}}}%
    \put(0.67201058,0.00407453){\color[rgb]{0,0,0}\makebox(0,0)[lt]{\lineheight{1.25}\smash{\begin{tabular}[t]{l}$0$\end{tabular}}}}%
    \put(0.66657981,0.12530513){\color[rgb]{0,0,0}\makebox(0,0)[lt]{\lineheight{1.25}\smash{\begin{tabular}[t]{l}$y$\end{tabular}}}}%
    \put(0.83500889,0.13696157){\color[rgb]{0,0,0}\makebox(0,0)[lt]{\lineheight{1.25}\smash{\begin{tabular}[t]{l}$w$\end{tabular}}}}%
    \put(0.74308995,0.06888043){\color[rgb]{0,0,0}\makebox(0,0)[lt]{\lineheight{1.25}\smash{\begin{tabular}[t]{l}\normalsize$\Omega$\end{tabular}}}}%
    \put(-0.00159041,0.08386349){\color[rgb]{0,0,0}\makebox(0,0)[lt]{\lineheight{1.25}\smash{\begin{tabular}[t]{l}$ $\end{tabular}}}}%
  \end{picture}%
\endgroup%
 
\end{center}

\end{Lemma}

\begin{proof}
Assume that $w=(x,x)\in\partial_0\Omega$ is a corner of $\partial_*\Omega$. Then, by Lemma~\ref{l:v_w}, there exists $v\in V_k$ such that $c_k(X_\Omega)=\|v\|_{\Omega}=\langle v,w\rangle$. The square $\Omega':=[0,x]^2$ is contained in $\Omega$, and satisfies $c_k(X_{\Omega'})=\|v\|_\Omega=c_k(X_\Omega)$. Since $X_\Omega$ maximizes $\chat_k$, we infer $X_\Omega=X_{\Omega'}=P(x,x)$.

Assume instead that $\partial_+\Omega$ contains a corner $w=(x,y)$. We claim that $y=y_0$ (and, in particular, that there are no other corners $w'=(x',y')\in\partial_+\Omega$ with $x'<x$). Indeed, once again by Lemma~\ref{l:v_w}, there exists $v_j=(j,k-j)\in V_k$ such that $c_k(X_\Omega)=\|v_j\|_{\Omega}=\langle v_j,w\rangle$. Since $x<y$, we have
\begin{align*}
&\langle v_h,w \rangle >c_k(X_\Omega),\quad\mbox{if }\ 0\leq h<j,\\
&\langle v_h,w \rangle < c_k(X_\Omega),\quad\mbox{if }\ j<h\leq k.
\end{align*}
This implies that the convex domain $\Omega':=(\R\times[0,y])\cap\Omega$, which is contained in $\Omega$, has the same capacity $c_k(X_{\Omega'})=c_k(X_{\Omega})$. Therefore, since $X_\Omega$ is a maximizer of $\chat_k$, we infer $\Omega'=\Omega$, and so $y=y_0$.

An analogous argument shows that, if $\partial_-\Omega$ contains a corner $w=(x,y)$, then $x=x_0$ (and then, in particular, there are no other corners $w'=(x',y')\in\partial_+\Omega$ with $x'>x$).
\end{proof}

By Lemma \ref{l:corners} it suffices to consider regions of the form
\[
\Omega=\Omega(\slp,s,t):=\{0 \leq y\leq 1+s \} \cap \{0 \leq x \leq 1+t \} \cap \{ 1- \slp (x-1) \leq y \}
\]
for $\slp \in [0,\infty]$, $s\in [0,\slp]$ and $t\in [0,1/\slp]$. The corners of $\partial_*\Omega$ are
\[
v^+:=(1-\tfrac{s}{\slp}, 1+s),\qquad v^-:=(1+t,1-\slp t).
\]
It will be convenient to parametrize $\slp$ as
\[
		\slp=\slp(i,r)=\frac{i+r}{k-(i+r)}
\]
where $(i,r) \in  \{0,1,\ldots,k-1\} \times [0,1) \cup \{(k,0)\}$. As above we set $v_j = (j,k-j)$, $j=1,\ldots,k$.

\begin{Lemma}\label{lem:maximizer} We have
\begin{align*}
\left\|v_j\right\|_{\Omega} & = \langle v_j,v^+  \rangle,
\qquad
\forall j\leq i,\\
\left\|v_{j}\right\|_{\Omega} & = \langle v_{j},v^-  \rangle,
\qquad
\forall j\geq i+1.
\end{align*}
\end{Lemma}
\begin{proof}
The maximum $\left\|v_j\right\|_{\Omega} = \max_{w\in \Omega}$ is attained at $v^+$ or $v^-$. Whether it is attained at $v^+$ or $v^-$ if $v^+\neq v^-$ depends on which side of the line $l$ through the origin perpendicular to the line through $v^+$ and $v^-$ the point $v_j$ lies. Since the slope of $l$ is $1/\slp$, the claim follows.
\end{proof}

\begin{Lemma} \label{lem:min_maximizer} We have
\begin{align*}
 \left\|v_j\right\|_{\Omega} & \geq \left\|v_i\right\|_{\Omega} = k \left( 1+ \frac{sr}{i+r}  \right), && \forall j\leq i,\\
 \left\|v_j\right\|_{\Omega} & \geq \left\|v_{i+1}\right\|_{\Omega} = k \left( 1+  \frac{t(1-r)}{k-(i+r)}  \right), && \forall j\geq i+1.
\end{align*}
\end{Lemma}
\begin{proof} The inequalities follow from the facts that the coordinates of $v_j$ sum up to $k$ and that the $x$-coordinate of $v^+$ ($v^-$) is not larger (smaller) than its $y$-coordinate.
For the equalities we compute using Lemma \ref{lem:maximizer}
\begin{align*} 
\left\|v_i\right\|_{\Omega} &= \left\langle v_i,v^+  \right\rangle =i-\frac s \slp i +k - i+ks -is =k \left( 1+ s\left(1- i  \frac{\slp+1}{k\slp}  \right) \right) \\
					&= k \left( 1+ s\left(1- i  \frac{1}{i+r}  \right) \right) = k \left( 1+ \frac{sr}{i+r}  \right) \\
\end{align*}
and
\begin{align*} 
\left\|v_{i+1}\right\|_{\Omega} &= \left\langle v_{i+1},v^+  \right\rangle = i+it+1+t+k-i-1-\slp tk+\slp it+\slp t \\
					&= k + t \left( 1+i-\slp(k-i-1) \right)  \\
					& = k+t  \frac{(1+i)\left(k-(i+r)\right)-\left(k-(i+1)\right)(i+r)}{k-(i+r)} \\
					& = k \left( 1+  \frac{t(1-r)}{k-(i+r)}  \right).
\qedhere
\end{align*}
\end{proof}
Using elementary geometry we compute the Euclidean area of $\Omega$ to be
\[
\area(\Omega) = 1+t+s-\frac{s^2}{2\alpha} - \frac{t^2\alpha}{2}.
\]
Hence, with $\vol(X_\Omega)= 2\, \area(\Omega)$ and Lemma \ref{lem:min_maximizer} we obtain
\begin{equation}\label{eqn:cond-2}
2\,\chat_k(X_\Omega)^2
=
\frac {c_k(X_\Omega)^2}{\area(\Omega)} 
= 
\frac{k^2 \left(1+ \min \left\{ \frac{sr}{i+r}, \frac{t(1-r)}{k-(i+r)} \right\}  \right)^2}{1+t+s-\frac{s^2}{2\alpha} - \frac{t^2\alpha}{2} }
\end{equation}
We claim that $2\,\chat_k(X_\Omega)^2 \leq k^2$ with equality if and only if $s=t=0$ (in which case $\Omega$ is independent of $\slp$), or $k=2$ and $X_\Omega=E(2,1)$. Showing this claim will complete the proof of Proposition \ref{mp:toric}(ii). The claim is equivalent to 
\begin{equation}\label{eqn:cond-1}
			\left(1+ \min \left\{ \frac{sr}{i+r}, \frac{t(1-r)}{k-(i+r)} \right\}  \right)^2\leq 1+t+s-\frac{s^2}{2\slp} - \frac{t^2\slp}{2} 
\end{equation}
with equality if and only if $s=t=0$, or $k=2$ and $X_\Omega=E(2,1)$. 

Since we have already treated the case of rectangles in Example \ref{ex:polydisk}, we assume that $0<i+r<k$, i.e.\ $\slp \in (0,\infty)$. In particular, we have $r<1$. Moreover, we can assume that $\frac{sr}{i+r}= \frac{t(1-r)}{k-(i+r)}$, i.e. 
\begin{equation}\label{eqn:cond0}
			t= \frac{r}{1-r} \frac{s}{\slp};
\end{equation}
indeed, otherwise we could decrease the volume $\vol(X_\Omega)$ by decreasing $t$ or $s$ without changing $c_k(X_\Omega)$. It suffices to show that we have a strict inequality in (\ref{eqn:cond-1}) in the case $s \neq 0$ unless $k=2$ and $X_\Omega=E(2,1)$. For $s\neq 0$ a strict inequality in (\ref{eqn:cond-1}) is (using (\ref{eqn:cond0})) equivalent to
\begin{align} \label{eqn:cond1}
					& 2s \frac{r}{i+r}+s^2\frac{r^2}{(i+r)^2} < s\left( 1+\frac{r}{(1-r)\slp} \right) - \frac{s^2}{2\slp} \left( 1+ \frac{r^2}{(1-r)^2}\right) \nonumber \\
					 \Leftrightarrow \, \, &\frac{s}{\slp} \left( \frac{r^2\slp}{(i+r)^2}+ \frac{(1-r)^2+r^2}{2(1-r)^2} \right) < 1+ \frac{r}{(1-r)\slp} - \frac{2r}{i+r}.
\end{align}
Note that $\frac{s}{\slp} \leq \min \left\{  1,\frac{1-r}{r\slp}\right\}$. Therefore, the sharp inequality (\ref{eqn:cond1}) is implied both by
\begin{equation}\label{eqn:cond2}
		 \frac{r^2\slp}{(i+r)^2}+ \frac{(1-r)^2+r^2}{2(1-r)^2}  < 1+ \frac{r}{(1-r)\slp} - \frac{2r}{i+r}
\end{equation}
and by
\begin{equation}\label{eqn:cond2b}
		 \frac{(1-r)r}{(i+r)^2}+ \frac{(1-r)^2+r^2}{2(1-r)r\slp}  < 1+ \frac{r}{(1-r)\slp} - \frac{2r}{i+r}.
\end{equation}
Let us first look at inequality (\ref{eqn:cond2}). 
\begin{Lemma}\label{lem:first_ineq} For $0<i+r<k$ inequality $(\ref{eqn:cond2})$ is equivalent to
\begin{equation}\label{eqn:cond3}
		 0<\left( (k-i)(k-i-2)+r  \right) r +(i+r)(k-i-r) \frac{1-2r}{2(1-r)}.
\end{equation}
\end{Lemma}
\begin{proof} Indeed, we have
\begin{align*}
				(\ref{eqn:cond2})\Leftrightarrow \, \,	& \frac{r^2}{(i+r)(k-i-r)}+ \frac{2r}{i+r}  <  \frac{r(k-i-r)}{(1-r)(i+r)}  +1- \frac{(1-r)^2+r^2}{2(1-r)^2}  \\
					 \Leftrightarrow \, \, &\left(r^2+2r(k-i-r)\right) (1-r) < r(k-i-r)^2+(i+r)(k-i-k) \frac{1-2r}{2(r-1)^2} \\
					\Leftrightarrow \, \, &0 < r(k^2+i^2-2ik - 2k+2i+r)+(i+r)(k-i-r) \frac{1-2r}{2(r-1)}  \\
						\Leftrightarrow \, \, & (\ref{eqn:cond3}).
\qedhere
\end{align*}
\end{proof}
Let us provide sufficient conditions for inequality (\ref{eqn:cond3}), and hence for (\ref{eqn:cond1}), to hold.
\begin{Lemma}\label{lem:case1}
For $0<i+r<k$ inequality $(\ref{eqn:cond3})$ holds if
\begin{itemize}\setlength\itemsep{2pt}
\item[$(i)$] $i \leq k-2$ and $r \in [0,\frac 1 2]$,
\item[$(ii)$] $i = k-1$ and $r \in [0,u_{k-1}:=\frac{k-1}{2k-1})$,
\item[$(iii)$] $i = 0$ and $r \in [0,u_{0}:=\frac{2k^2-3k}{2k^2-2k-1})$.
\end{itemize}
\end{Lemma}
\begin{proof} Statement (i) is immediate from Lemma \ref{lem:first_ineq}. For $i=k-1$, inequality (\ref{eqn:cond3}) is equivalent to
\begin{align*}
				& 0<\left( -1+r  \right) r +(k-1+r)(1-r) \frac{1-2r}{2(1-r)}  \\
					 \Leftrightarrow \, \, & 0< 2r^2-2r+k-1+r-2rk+2r-2r^2 \\
										 \Leftrightarrow \, \, & r < u_{k-1}=\frac{k-1}{2k-1},
\end{align*}
which shows (ii). For $i=0$ we have $r>0$ and inequality (\ref{eqn:cond3}) is equivalent to
\begin{align*}
				& 0<\left( k(k-2)+r  \right) r +r(k-r) \frac{1-2r}{2(1-r)}  \\
					 \Leftrightarrow \, \, & (2r-1)(k-r)<2 (1-r)(k(k-2)+r)\\
										 \Leftrightarrow \, \, & 2rk-2r^2-k+r< 2k^2-4k+2r-2rk^2+4rk-2r^2\\
										\Leftrightarrow \, \, & r< u_0 =\frac{2k^2-3k}{2k^2-2k-1},
\end{align*}
which shows (iii) and completes the proof of the lemma.
\end{proof}

Let us now look at inequality (\ref{eqn:cond2b}).

\begin{Lemma}\label{lem:second_ineq} For $0<i+r<k$ inequality $(\ref{eqn:cond2b})$ is equivalent to
\begin{equation}\label{eqn:cond4}
		 r < i^2 + \frac{(i+r)^2(2r-1)}{2(1-r)r\slp}.
\end{equation}
\end{Lemma}
\begin{proof}
Indeed, we have
\begin{align*}
				(\ref{eqn:cond2b})\Leftrightarrow \, \,	&  \frac{r}{(i+r)^2} \left( (1-r)+2(i+r) \right) < 1 + \frac{1}{2(1-r)r\slp} \left(2r^2-(1-r)^2-r^2\right) \\
					 \Leftrightarrow \, \, & r-r^2+2ir+2r^2<i^2+2ir+r^2+ \frac{(i+r)^2(2r-1)}{2(1-r)r\slp} \\
					 \Leftrightarrow \, \, & (\ref{eqn:cond4}).
\qedhere
\end{align*}
\end{proof}
Now we provide sufficient conditions for inequality (\ref{eqn:cond3}) and hence further conditions for (\ref{eqn:cond1}) to hold.
\begin{Lemma}\label{lem:case2}
For $0<i+r<k$, inequality $(\ref{eqn:cond4})$ holds if
\begin{itemize}\setlength\itemsep{2pt}
\item[$(i)$] $i >0 $ and $r \in [\frac 1 2,1)$,
\item[$(ii)$] $i = 0$ and $r \in (l_{0}:=\frac{k}{2k-1},1)$,
\item[$(iii)$] $i = k-1$ and $r \in (l_{k-1}:=\frac{k-1}{2k^2-2k-1},1)$.
\end{itemize}
\end{Lemma}
\begin{proof}
Statement (i) is immediate from Lemma \ref{lem:second_ineq}. For $i=0$ we have $r>0$ and inequality (\ref{eqn:cond3}) is equivalent to
\begin{align*}
				& r <  \frac{r^2(2r-1)(k-r)}{2(1-r)r^2}= \frac{(2r-1)(k-r)}{2(1-r)}.  \\
					 \Leftrightarrow \, \, & 2r-2r^2< 2kr-k-2r^2+r \\
										 \Leftrightarrow \, \, & l_{0}=\frac{k}{2k-1}<r,
\end{align*}
which shows (ii). For $i=k-1$ inequality (\ref{eqn:cond3}) is equivalent to
\begin{align*}
				& r<(k-1)^2+\frac{(k-1+r)^2(1-r)(2r-1)}{2(1-r)r(k-1+r)}  \\
					 \Leftrightarrow \, \, & r<(k-1)^2+\frac{(k-1+r)(2r-1)}{2r}  \\
										 \Leftrightarrow \, \, & 2r^2<2rk^2-4rk+2r+2rk-2r+2r^2-k+1-r\\
										\Leftrightarrow \, \, & l_{k-1}=\frac{k-1}{2k^2-2k-1} < r,
\end{align*}
which shows (iii) and completes the proof of the lemma.
\end{proof}

It remains to show that the conditions in Lemma \ref{lem:case1} and Lemma \ref{lem:case2} imply a strict inequality in (\ref{eqn:cond-1}) unless $k=2$, and to study the equality case for $k=2$. By part $(i)$ of Lemma \ref{lem:case1} and Lemma \ref{lem:case2} it suffices to look at the cases $i=k-1$ and $i=0$. This will be done in the subsequent two lemmas. 

\begin{Lemma}\label{lem:case_n-1}
We have $l_{k-1}\leq u_{k-1}$, and $l_{k-1} < u_{k-1}$ unless $k=2$. In the case $k=2$ we get an equality in $(\ref{eqn:cond-1})$ if and only if $\slp=\slp(i,r)=2$, $s=\slp=2$ and $t= \frac 1 2$, in which case $X_\Omega=E(2,1)$.
\end{Lemma}
\begin{proof} A quick computation using the expressions for $l_{k-1}$ and $u_{k-1}$ from Lemma \ref{lem:case1} and Lemma \ref{lem:case2} (still assuming $k\geq 2$) confirms that $l_{k-1}\leq u_{k-1}$ with equality if and only if $k=2$. Hence, for the equality discussion it suffices to consider the case $k=2$, $i=1$ and $r=l_{1}= u_{1}=\frac{1}{3}$. Then $\slp=\slp(i,r)=2$. Plugging this into (\ref{eqn:cond1}) with an equality instead of an inequality implies that $s=2$ and by (\ref{eqn:cond0}) thus $t=\frac 1 2$.
\end{proof}

\begin{Lemma}\label{lem:case_0}
We have $l_{0}\leq u_{0}$, and $l_{0} < u_{0}$ unless $k=2$. In the case $k=2$, $r=l_{1}= u_{1}=\frac{2}{3}$ and $i=0$ we get an equality in $(\ref{eqn:cond-1})$ if and only if $\slp=\slp(i,r)=\frac 1 2$, $s=\slp=\frac 1 2$ and $t=2$,  in which case $X_\Omega=E(2,1)$.
\end{Lemma}
\begin{proof} A quick computation using the expressions for $l_{0}$ and $u_{0}$ from Lemma \ref{lem:case1} and Lemma \ref{lem:case2} (still assuming $k\geq 2$) shows that $l_{0}\leq u_{0}$ is equivalent to $3k\leq k^2 +2$. It follows that equality holds if and only if $k=2$. Hence, for the equality discussion it suffices to consider the case $k=2$, $i=0$ and $r=l_{1}= u_{1}=\frac{2}{3}$. Then $\slp=\slp(i,r)=\frac 1 2$. Plugging this into (\ref{eqn:cond1}) with an equality instead of an inequality implies that $s=\frac 1 2$ and by (\ref{eqn:cond0}) thus $t=2$.
\end{proof}

This completes the proof of Proposition \ref{mp:toric}(ii).

\appendix

\section{The Fadell-Rabinowitz index}
\label{a:Fadell_Rabinowitz}
Consider the universal bundle $ES^1\to BS^1$. The classifying space $BS^1$ has cohomology ring $H^*(BS^1;\Q)=\Q[e]$, where $e$ is a generator of $H^2(BS^1;\Q)$. From now on, all cohomology rings will be assume to have rational coefficients, and we will suppress $\Q$ from the notation.
For each non-empty topological space $X$ equipped with an $S^1$ action, we consider the $S^1$-equivariant cohomology 
\[H^*_{S^1}(X):=H^*(X\times_{S^1} ES^1).\] 
Here, $X\times_{S^1} ES^1:=(X\times ES^1)/S^1$ and the circle $S^1$ acts diagonally on the product. If $Y\subset X$ is an $S^1$-invariant subspace, the $S^1$-equivariant relative cohomology is defined by
\[H^*_{S^1}(X,Y):=H^*(X\times_{S^1} ES^1,Y\times_{S^1} ES^1).\]
Let $\pi:X\times_{S^1} ES^1\to BS^1$, $\pi([x,y])=[y]$ be the quotient-projection. With an abuse of notation, we will still denote by $e$ the cohomology class $\pi^*e\in H^2_{S^1}(X)$. The \emph{Fadell-Rabinowitz index} of $X$, first introduced in \cite{Fadell:1978aa}, is defined as\footnote{In the literature the Fadell-Rabinowitz index is occasionally defined as $\FR(X)-1$.}
\begin{align*}
\FR(X)=\inf\big\{k\geq0\ \big|\ e^{k}=0\mbox{ in }H_{S^1}^{*}(X)\big\},
\end{align*}
with the usual convention $\inf\varnothing=\infty$. Its elementary properties are the following:
\begin{itemize}
\setlength\itemsep{5pt}

\item \textbf{(Non-triviality)} $\FR(X)=0$ if and only if $X=\varnothing$.

\item \textbf{(Monotonicity)} $\FR(X)\leq\FR(Y)$ if there exists an $S^1$-equivariant continuous map $f:X\to Y$.

\item \textbf{(Subadditivity)} $\FR(X)\leq\FR(Y)+\FR(Z)$ if $X=\interior(Y)\cup\interior(Z)$, where the interior is taken with respect to the subspace topology.

\end{itemize}

\begin{Remark}
In the literature, the cohomology employed in the definition of the Fadell-Rabinowitz index is often the Alexander-Spanier one \cite[Section~6.4]{Spanier:1995aa}, which has a suitable continuity property \cite[Appendix A.8]{Hofer:1994aa}: if a topological space is metrizable, 
the cohomology of any subspace is isomorphic to the direct limit of the cohomology of its open neighborhoods. We shall equivalently employ the more common singular cohomology, which does not satisfy such a strong continuity property, and circumvent the issue by taking open neighborhoods in statements such as the Lusternik-Schnirelmann property of spectral invariants (Section~\ref{ss:Clarke_spectral_invariants}), Proposition~\ref{p:Borsuk_Ulam} below, and in the definition of the Ekeland-Hofer index (Section~\ref{ss:def_EH_capacities}).
\end{Remark}

In the proof of Proposition~\ref{p:capacity<spectral}, we need some properties of the Fadell-Rabinowitz index that are essentially established in the original source \cite{Fadell:1978aa} and in Fadell-Husseini-Rabinowitz \cite{Fadell:1982aa}, but stated under slightly different assumptions or contained within the proofs of some statements. We provide a self-contained account here for the reader's convenience.
For any subset $B$ of a vector space, we denote 
\[B_*:=B\setminus\{0\}.\] 
Throughout this appendix, we denote by $S^1$ the unit circle in $C$.
We shall consider vector spaces $\C^m\oplus\R^n$, with $m\geq1$ and $n\geq2$, equipped with the $S^1$ action
\begin{align*}
e^{it}\cdot (z_1,...,z_{m},x_1,...,x_n) = ( e^{i k_1 t}z_1,..., e^{i k_{m}t}z_{m}, x_1,..., x_n),
\end{align*}
for some positive integers $k_1,...,k_m$. The fixed point set of this action is the vector subspace $\R^n\equiv\{0\}\oplus\R^n$. On $\C^m\oplus\R^n$, we also have the action of the multiplicative group $(0,\infty)$ given by the scalar multiplication, i.e.
\begin{align*}
c\cdot(z_1,...,z_m,x_1,...,x_n)=(cz_1,...,cz_m,cx_1,...,cx_n).
\end{align*}
Overall, the $S^1$ and $(0,\infty)$ actions give a $\C_*\equiv S^1\times(0,\infty)$ action.
Direct sum of vector spaces as above will be implicitly equipped with the product $\C_*$ action.
We say that a subset of $\C^m\oplus\R^n$ is $\C$-invariant when it is $\C_*$-invariant and contains the origin. The first lemma is a special case of \cite[Prop.~4.3]{Fadell:1978aa}.

\begin{Lemma}
\label{l:stabilization}
For each $\C$-invariant subset $Y\subset\C^m$, we have
\[\FR((Y\oplus\C)_*)=\FR(Y_*)+1.\]
\end{Lemma}

\begin{proof}
We set $X:=Y\oplus\C$. Since $X_* = (Y_*\oplus\C)\cup (Y\oplus\C_*)$, by the subadditivity of the Fadell-Rabinowitz index we have
\begin{align*}
 \FR(X_*) 
 & \leq
 \FR(Y_*\oplus\C)+ \FR(Y\oplus\C_*)\\
 & =
 \FR(Y_*)+ \FR(\C_*)= \FR(Y_*)+1.
\end{align*}
Since $Y_*\subset X_*$, we have
$\FR(Y_*)\leq \FR(X_*)$.
Assume by contradiction that $k:=\FR(Y_*)= \FR(X_*)$. We consider the commutative diagram
\[
\begin{tikzcd}
... 
\arrow[r] 
& H^{2k-1}(X_*) 
\arrow[r,"\pi_*"]
\arrow[d,"i^*=0"']
& 
H^{2k-2}_{S^1}(X_*) 
\arrow[r,"\smallsmile e"]
\arrow[d]
& 
H^{2k}_{S^1}(X_*) 
\arrow[r] 
\arrow[d]
& 
...\\
... 
\arrow[r] 
& H^{2k-1}(Y_*) 
\arrow[r,"\pi'_*"]
& 
H^{2k-2}_{S^1}(Y_*) 
\arrow[r,"\smallsmile e"]
& 
H^{2k}_{S^1}(Y_*) 
\arrow[r] 
& 
...
\end{tikzcd}
\]
where the vertical homomorphisms are induced by the inclusions, and both rows are Gysin exact sequences. Notice that $i^*=0$, since the inclusion $i:Y_*\hookrightarrow X_*$ is homotopic to a constant. Since 
$k=\FR(X_*)$, we have $e^{k-1}=\pi_*(\mu)$ for some $\mu\in H^{2k-1}(X_*)$, and therefore $e^{k-1}=\pi'_* i^*(\mu)=0$ in $H^{2k-2}_{S^1}(Y_*)$, contradicting the fact that $\FR(Y_*)=k$.
\end{proof}

Since $S^1$ acts trivially on $\R^n$, we have the K\"unneth isomorphism 
\[H^*_{S^1}(\R^n_*)=H^*(BS^1\times\R^n_*)\cong H^*(BS^1)\otimes H^*(\R^n_*) .\]
In particular, $H^*(\R^n_*)\equiv H^0(BS^1)\otimes H^*(\R^n_*)$ is naturally a subgroup of $H^*_{S^1}(\R^n_*)$. 

\begin{Lemma}\label{l:relative_FR}
Let $Y\subset\C^{m}$ be a $\C$-invariant subset,  $\mu$  a generator of $H^{n-1}(\R^n_*)$, and $\delta:H_{S^1}^{*}(\R^n_* ) \to H_{S^1}^{*+1}((Y\oplus\R^n)_*,\R^n_*)$ the connecting homomorphism. Then 
\begin{align*}
 \FR(Y_*)=\inf\big\{k\geq 0\ \big|\ e^k\smallsmile\delta(\mu)=0\mbox{ in }H_{S^1}^*((Y\oplus\R^n)_*,\R^n_*)\big\}.
\end{align*}
\end{Lemma}

\begin{proof}
We set $X:=Y\oplus\R^n$, decompose $X_*$ as the union
$X_*=(Y_*\oplus\R^n) \cup (Y\oplus\R_*^n)$,
and consider the associated Mayer-Vietoris sequence in $S^1$-equivariant cohomology
\begin{align*}
...
\toup 
H^{*}_{S^1}(X_*) 
\toup^{a^*}  
H^{*}_{S^1}(Y_*\oplus\R^n)\oplus H^{*}_{S^1}(Y\oplus\R^n_*)
\toup^{b^*}  
H^{*}_{S^1}(Y_*\oplus\R^n_*) 
\toup
... 
\end{align*}
Since $Y$ and $\R^n$ are $S^1$-equivariantly contractible spaces, we have $H^{*}_{S^1}(Y_*\oplus\R^n)\cong
 H^{*}_{S^1}(Y_*)$ and $H^{*}_{S^1}(Y\oplus\R^n_*)\cong H^{*}_{S^1}(\R^n_*)$. Moreover, since $S^1$ acts trivially on $\R^n$, we have $H^{*}_{S^1}(Y_*\oplus\R^n_*) \cong 
H^{*}_{S^1}(Y_*)\otimes H^*(\R^n_*)$. Therefore, the Mayer-Vietoris sequence can be rewritten as
\begin{align*}
 ...
\toup 
H^{*}_{S^1}(X_*) 
\toup^{a^*}  
H^{*}_{S^1}(Y_*)\oplus H^{*}_{S^1}(\R^n_*)
\toup^{b^*}  
H^{*}_{S^1}(Y_*)\otimes H^*(\R^n_*) 
\toup
... 
\end{align*}
Let us now consider the long exact sequence of the pair $\R^n_*\subset X_*$, which reads
\begin{equation*}
...\toup
H^*_{S^1}(X_*) \toup^{j^*} 
H^*_{S^1}(\R^n_*) \toup^{\delta} 
H^{*+1}_{S^1}(X_*,\R^n_*) \toup
...
\end{equation*}
Let $\mu$ be a generator of $H^{n-1}(\R^n_*)$, and $k:=\FR(Y_*)$. We have a non-zero cohomology class
\[
e^{k-1}\otimes\mu\in H^{*}(BS^1)\otimes H^{*}(\R^n_*) = H_{S^1}^{*}(\R^n_*) 
\]

Let us assume by contradiction that  $\delta(e^{k-1}\otimes\mu)=0$, so that there exists a cohomology class $\nu\in H_{S^1}^{*}(X_*)$ such that $j^*\nu=e^{k-1}\otimes\mu$. In the above Mayer-Vietoris sequence, we have 
$a^*(\nu) = ( \nu|_{Y_*},e^{k-1}\otimes\mu )$,
and therefore
\[0 = b^*( \nu|_{Y_*},e^{k-1}\otimes\mu ) = \nu|_{Y_*}\otimes 1 - e^{k-1}|_{Y_*}\otimes\mu,\]
which gives a contradiction, since $e^{k-1}|_{Y_*}\otimes\mu\neq 0$. This shows that
\begin{align*}
e^{k-1}\smallsmile\delta(1\otimes\mu) = \delta(e^{k-1}\smallsmile\mu)\neq 0.
\end{align*}

On the other hand, since $\FR(Y_*)=k$, we have
\[
b^*(0,e^{k}\otimes\mu) = - e^{k}|_{Y_*} \otimes \mu = 0.
\]
Therefore, there exists $\eta\in H^{*}_{S^1}(X_*)$ such that $a^*(\eta)=(0,e^{k}\otimes\mu)$. This implies $j^*(\eta)=e^{k}\otimes\mu$, and we conclude
\[
e^{k}\smallsmile\delta(1\otimes\mu) = \delta(e^{k}\smallsmile\mu) = \delta j^*(\eta) = 0.
\qedhere
\]
\end{proof}

The last statement is a slight generalization of \cite[Prop.~2.2]{Fadell:1982aa}.

\begin{Prop}
\label{p:Borsuk_Ulam}
Let $Y\subset\C^p$ be a $\C$-invariant subset, $B\subset\C^p\oplus\R^n$ an $S^1$-invariant compact neighborhood of the origin, and $\psi:Y\oplus\R^n\to \C^{q}\oplus\R^n$ an $S^1$-equivariant continuous map such that $\psi(0,x)=(0,x)$ for all $x\in\R^n$. If $\FR(Y_*)>q$, then any $S^1$-invariant open subset $U\subset Y\oplus\R^n$ containing the intersection $\psi^{-1}(0)\cap\partial B$ has Fadell-Rabinowitz index $\FR(U)\geq\FR(Y_*)-q$.
\end{Prop}

\begin{proof}
We set $k:=\FR(Y_*)$ and $X:=Y\oplus\R^n$. Assume by contradiction that there exists an $S^1$-invariant open subset $U\subset X_*$ containing $Z:=\psi^{-1}(0)\cap\partial B$ such that $\FR(U)<k-q$. In particular, since $\FR(U)$ is finite, $U\cap\R^n=\varnothing$. Let $V\subset X_*\setminus Z$ be an $S^1$-invariant open subset containing $X\cap\partial B\setminus U$ and that is sufficiently small so that $X\cap \psi^{-1}(0)=\varnothing$. Finally, we set $W:=U\cup V$, which is an $S^1$-invariant open subset of $X_*$ containing $X\cap\partial B$.

We defined the $S^1$-invariant subsets 
$X':=(X_*\cap B)\cup W$ and  $X'':=(X_*\setminus X')\cup W$, 
which are both open subsets of $X_*$, have union $X'\cup X''=X_*$ and intersection $X'\cap X''=W$. 
Notice that, for each $x\in X_*$ there exists $\epsilon>0$ small enough so that $\epsilon x\in X'$ and $\epsilon^{-1}x\in X''$. This readily implies that both inclusions 
\[(X',X'\cap\R^n_*)\hookrightarrow(X_*,\R^n_*),\qquad (X'',X''\cap\R^n_*)\hookrightarrow(X_*,\R^n_*)\]
admit an $S^1$-equivariant homotopic right inverse, and therefore induce monomorphisms
\begin{align*}
 H^*_{S^1}(X_*,\R^n_*)&\hookrightarrow H^*_{S^1}(X',X'\cap\R^n_*),
 \\
 H^*_{S^1}(X_*,\R^n_*)&\hookrightarrow H^*_{S^1}(X'',X''\cap\R^n_*).
\end{align*}
This, together with the Mayer-Vietoris exact sequence
\begin{align*}
... \to
H^*_{S^1}(X_*,\R^n_*)
 \to
  \begin{array}{c}
    H^*_{S^1}(X',X'\cap\R^n_*)\\ \oplus\\ H^*_{S^1}(X'',X''\cap\R^n_*) 
  \end{array} 
 \to
 H^*_{S^1}(W,W\cap\R^n_*)
 \to ...
\end{align*}
readily implies that the inclusion induces a monomorphism
\begin{align}
\label{e:injectivity_piercing}
H^*_{S^1}(X_*,\R^n_*)\hookrightarrow H^*_{S^1}(W,W\cap\R^n_*).
\end{align}

Analogously, since both inclusions $X'\cap\R^n_*\hookrightarrow\R^n_*$ and $X''\cap\R^n_*\hookrightarrow\R^n_*$ admit a homotopic right inverse, they induce monomorphisms in cohomology, and the Mayer-Vietoris sequence
\begin{align*}
... \to
H^*(\R^n_*)
\to
H^*(X'\cap\R^n_*)\oplus H^*(X''\cap\R^n_*) 
\to
H^*(W\cap\R^n_*)
\to ...
\end{align*}
readily implies that the inclusion induces a monomorphism
\begin{align}
\label{e:injectivity_piercing_0}
H^*(\R^n_*)\hookrightarrow H^*(W\cap\R^n_*).
\end{align}

Consider the commutative diagram
\[
\begin{tikzcd}
H^{*}_{S^1}(\R^n_*)
\arrow[r,"\delta"]\arrow[d]&
H_{S^1}^{*+1}(X_*,\R^n_*)\arrow[d]\\
H^{*}_{S^1}(W\cap\R^n_*)
\arrow[r,"\delta"]\arrow[d,"\cong"']&
H_{S^1}^{*+1}(W,W\cap\R^n_*)\arrow[d]\\
H^{*}_{S^1}(V\cap\R^n_*)
\arrow[r,"\delta"]&
H_{S^1}^{*+1}(V,V\cap\R^n_*)\\
H^{*}_{S^1}(\R^n_*)
\arrow[r,"\delta"]\arrow[u]&
H_{S^1}^{*+1}((\C^q\oplus\R^n)_*,\R^n_*)\arrow[u,"\psi^*"']
\end{tikzcd}
\]
Here, the horizontal arrows are connecting homomorphisms, and the unspecified vertical arrows are homomorphisms induced by the inclusions.
Let $\mu$ be a generator of $H^{n-1}(\R^n_*)\subset H^{n-1}_{S^1}(\R^n_*)$. By the injectivity of the homomorphism~\eqref{e:injectivity_piercing_0}, in the above diagram $\mu$ is mapped to a non-zero element of $H_{S^1}^{*}(V\cap\R^n_*)=H_{S^1}^{*}(W\cap\R^n_*)$, which we still denote by $\mu$.
Since $\FR(Y_*)=k$, Lemma~\ref{l:relative_FR} implies that 
\[e^{k-1}\smallsmile\delta(\mu)\neq0 \mbox{ in } H_{S^1}^*(X_*,\R^n_*),\]
and therefore, by the injectivity of the homomorphism~\eqref{e:injectivity_piercing},
\begin{align}
\label{e:crucial_non_vanishing_FR_app}
e^{k-1}\smallsmile\delta(\mu)\neq0 \mbox{ in } H_{S^1}^*(W,W\cap\R^n_*).
\end{align}
Lemma~\ref{l:stabilization} implies $\FR(\C^q_*)=q$. Therefore, another application of Lemma~\ref{l:relative_FR} implies that
\begin{align*}
 e^{q}\smallsmile\delta(\mu)=0 \mbox{ in } H_{S^1}^*((\C^q\oplus\R^n)_*,\R^n_*).
\end{align*}
By applying $\psi^*$, we obtain $e^{q}\smallsmile\delta(\mu)=0 \mbox{ in } H_{S^1}^*(V,V\cap\R^n_*)$, and therefore
\begin{align*}
 e^{q}\smallsmile\delta(\mu)\in H_{S^1}^*(W,V).
\end{align*}
Since $\FR(U)\leq k-q-1$, we have
\begin{align*}
 e^{k-q-1}\in H^{*}_{S^1}(W,U).
\end{align*}
Finally
\begin{align*}
e^{k-1}\smallsmile\delta(\mu) = e^{k-q-1}\smallsmile (e^q\smallsmile\delta(\mu)) \in H^{*}_{S^1}(W,U\cup V)=H^{*}_{S^1}(W,W)=0,
\end{align*}
which contradicts~\eqref{e:crucial_non_vanishing_FR_app}.
\end{proof}

\subsection*{Acknowledgments} 
We are indebted with Alberto Abbondandolo, who first noticed while working with Lange and Mazzucchelli on \cite{Abbondandolo:2022aa} that the higher capacity ratios of the polydisk $P(1,1)$ were higher than the ones of 4-dimensional ellipsoids, and who suggested to improve an earlier formulation of Theorem~\ref{mt:local_max} beyond the space of smooth convex bodies. We also thank Jean Gutt, Felix Schlenk, and Jean-Claude Sikorav for some helpful conversations on symplectic capacities.

\bibliography{_biblio}

\end{document}